\documentclass[11pt,a4paper]{article}
\usepackage[utf8]{inputenc}
\usepackage{amsmath}
\usepackage{amsfonts}
\usepackage{amssymb}
\usepackage{amsthm}
\usepackage{setspace}
\usepackage{hyperref}
\usepackage{graphicx}
\usepackage{stmaryrd}
\usepackage{mathtools}
\usepackage{subcaption}
\usepackage[T1]{fontenc}
\usepackage{xcolor}
\usepackage{float}
\usepackage{tikz-cd}
\usepackage[noadjust]{cite}
\usepackage{cite}
\usepackage{ragged2e}
\usepackage[margin=1.2in]{geometry}
\graphicspath{ {./images/} }
\newtheorem{theorem}{Theorem}[section]
\newtheorem{proposition}[theorem]{Proposition}
\newtheorem{lemma}[theorem]{Lemma}
\newtheorem{corollary}[theorem]{Corollary}

\theoremstyle{definition}
\newtheorem{definition}[theorem]{Definition}
\newtheorem{example}[theorem]{Example}
\newtheorem{algo}[theorem]{Algorithm}
\newtheorem{convention}[theorem]{Convention}
\newtheorem{remark}[theorem]{Remark}

\newtheorem{question}[theorem]{Question}

\newcommand{\Aut}{\mathrm{Aut}}
\newcommand{\Out}{\mathrm{Out}}

\newcommand{\Fix}{\mathrm{Fix}}
\newcommand{\im}{\mathrm{Im}}
\newcommand{\Min}{\mathrm{Min}}
\newcommand{\Stab}{\mathrm{Stab}}

\title{Homomorphisms between XL-type Artin groups}
\author{Mart\'in Blufstein, Alexandre Martin, and Nicolas Vaskou}
\date{}

\begin{document}

\maketitle 

\begin{abstract}
\centering \justifying We study homomorphisms between XL-type Artin groups and show that, in a suitable sense, a generic Artin group is both hopfian and co-hopfian.

For XL-type Artin groups over complete graphs, we describe all possible homomorphisms with sufficiently large image, and prove in particular that such groups are both hopfian and co-hopfian. For Artin groups over general graphs with all labels at least $6$, we characterise in terms of the presentation graph exactly when these groups are co-hopfian, as well as when they have a finite outer automorphism group. When in addition the presentation graph has no cut-vertex, we show that their automorphism group is finitely generated and we provide a generating set.  
\end{abstract}

\noindent \rule{7em}{.4pt}\par

\small



\normalsize

\section{Introduction}

\paragraph{Motivation.}  Artin groups are an actively studied class of groups, with connections to braid groups and Coxeter groups. Recall that a \textbf{presentation graph} is a finite simplicial graph $\Gamma$ where every edge between vertices $s, t$ is labelled by an integer $m_{st} \geq 2$.
Given a presentation graph $\Gamma$, the associated  \textbf{Artin group} $A_\Gamma$ is the group given by the following presentation:
\[A_\Gamma \coloneqq \langle s \in V(\Gamma) \mid \Pi(s, t; m_{st}) = \Pi(t, s; m_{st})  ~ \mbox{ whenever } s,t \mbox{ are adjacent in }\Gamma \rangle, \]
where $\Pi(x, y; k)$ denotes the alternating product of $xyxy\cdots$ with $k$ letters. The \textbf{rank} of an Artin group $A_{\Gamma}$ with presentation graph $\Gamma$ is defined as the number $|V(\Gamma)|$ of standard generators for that presentation.

Given a group, or a family of groups, the study of the isomorphisms and homomorphisms between groups in this family is a natural problem. For Artin groups, besides being interesting in their own right, such questions are also relevant for other open problems:

\medskip

\textit{Isomorphisms and automorphism groups.} Since an Artin group is given by means of a presentation graph, i.e. a choice of presentation, the Isomorphism Problem asks when two presentation graphs yield isomorphic Artin groups.
This problem is still wide open for general Artin groups. It is also still open for Coxeter groups, although much more is known in that case (see \cite{santosregoschwer2024} for a recent survey).
In the case of Artin groups, this has been solved only for a few classes, such as right-angled Artin groups \cite{droms1987isomorphisms}, spherical-type Artin groups \cite{paris2003artin}, and large-type Artin groups \cite{vaskou2023isomorphism}.

Once one knows when two Artin groups are isomorphic, the descriptions of all possible isomorphisms between them boils down to a study of the automorphism groups of individual Artin groups.
For right-angled Artin groups, a generating set is known by work of Laurence and Servatius \cite{servatius1989automorphisms, laurence1995generating}. For general Artin groups, the automorphism group is known only for a few families: certain large-type Artin groups \cite{crisp2005automorphisms, an2022automorphism, vaskou2023automorphisms,paris2024endomorphisms, paris2024endomorphismsB}, as well as certain spherical-type and affine-type Artin groups \cite{gilbert2000tree, charney2005automorphism, crisp2005automorphisms, soroko2021artin, castel2023endomorphism}.
\medskip

\textit{Endomorphisms, and the hopf / co-hopf property.} Very little is currently known about the general endomorphisms of an Artin group, with only a few cases being known \cite{paris2024endomorphisms, paris2024endomorphismsB}.
However, more results are known for certain classes of endomorphisms, such as monomorphisms and epimorphisms.

An important particular case is to determine which Artin groups are \textbf{co-hopfian}, i.e. which Artin groups are not isomorphic to one of their strict subgroups.
Many Artin groups are known not to be co-hopfian, such as non-trivial right-angled Artin groups and spherical-type Artin groups.
Before this article, only Artin groups of affine type $\widetilde{A}_n$ and $\widetilde{C}_n$ are known to be co-hopfian \cite{bell2007injections}.

Another related problem is to determine whether Artin groups are \textbf{hopfian}, i.e. whether they cannot be isomorphic to a non-trivial quotient.
This is connected to the problem of determining which Artin groups are residually finite, as finitely generated residually finite groups are hopfian.
At the moment, this is only known for a few families including right-angled Artin groups, spherical-type Artin groups \cite{cohenwaleslinearity2002}, some even Artin groups \cite{blascogarcia2019even}, and some two-dimensional Artin groups  \cite{jankiewicz2022RF}. 
To our knowledge, before the results presented in this article, there were no examples of hopfian Artin groups for which we do not already know that they are residually finite. (In a forthcoming article \cite{MangioniSisto}, Mangioni and Sisto prove the Hopf property for another class of Artin groups that includes in particular the even Artin groups of XL-type.)

\paragraph{Statement of results.} In this article, we study the homomorphisms and isomorphisms between certain families of Artin groups.
We will focus on \textbf{extra-large-type} Artin groups (hereafter simply \textbf{XL-type}), i.e. those for which all labels satisfy $m_{st} \geq 4$.
Our results concern in particular the following two families: 
\begin{itemize}
    \item XL-type Artin groups that are \textbf{free-of-infinity} (i.e. the presentation graph $\Gamma$ is a complete graph), and
    \item Artin groups such that all labels satisfy $m_{st}\geq 6$ (which we will simply refer to as \textbf{XXXL-type} Artin groups, in accordance with existing terminology). 
\end{itemize}

In \cite{goldsborough2023random}, Goldsborough--Vaskou introduced a model of random Artin groups, that depends on a choice of growth parameter.
In this model, the aforementioned two classes are ``generic” for a suitable choice of parameter.
In particular, a ``random Artin group'' satisfies the results listed below.
This contrasts with other well-studied families of Artin groups that are not generic, such as spherical-type Artin groups, right-angled Artin groups, or the large-type triangle-free Artin groups studied by Crisp and An-Cho \cite{crisp2005automorphisms, an2022automorphism}.

\paragraph{Free-of-infinity Artin groups.} Let us start with free-of-infinity Artin groups. In this case, we can actually describe \textit{all} homomorphisms that have a sufficiently large image.

\begin{theorem}\label{thm:structure_hom}
    Let $f: A_\Gamma \rightarrow A_{\Gamma'}$ be a homomorphism between XL-type free-of-infinity Artin groups with rank at least $3$, such that the image of $f$ is not cyclic or contained in a subgroup of $A_{\Gamma'}$ that virtually splits non-trivially as a  direct product. Let $\Gamma_1 \subseteq \Gamma$ be a minimal induced subgraph such that $f(A_{\Gamma_1}) = f(A_\Gamma)$. Then there exist an induced subgraph $\Gamma_1' \subseteq \Gamma'$, an isomorphism  $\iota: \Gamma_1 \rightarrow \Gamma_1'$ of unlabelled graphs, an element $g \in A_{\Gamma'}$, and an integer $\varepsilon =\pm 1$ such that the following holds: 
    \begin{itemize}
        \item for every distinct $s, t \in V(\Gamma_1)$, we have that $m_{st}$ is a multiple of $m_{\iota(s)\iota(t)}$, and
        \item for every $s\in V(\Gamma_1)$ we have $f(s)=g\iota(s)^\varepsilon g^{-1}$.
    \end{itemize}   
    In particular, the image of $f$ is the parabolic subgroup $gA_{\Gamma_1'}g^{-1}$ of $A_{\Gamma'}$. 
\end{theorem}

We obtain the following corollary, which in particular provides the first  ``generic'' family of Artin groups known to be  co-hopfian:

\begin{corollary}\label{thm:hopf_cohopf}
    Let $A_\Gamma$ be an XL-type free-of-infinity Artin group of rank at least $3$. Then $A_\Gamma$ is hopfian and co-hopfian.
\end{corollary}

An important open problem related to the Isomorphism Problem is to determine whether the rank of an Artin group $A_{\Gamma}$ is an isomorphism invariant.
In the case of large-type Artin groups, this is indeed the case and follows from Vaskou's solution to the Isomorphism Problem for that class (\cite{vaskou2023isomorphism}).
However, a natural related problem is to understand how the rank behaves under other types of endomorphisms, such as injective and surjective maps.
We obtain the following monotonicity result: 

\begin{corollary}\label{cor:surj_inj}
    Let $f: A_\Gamma \rightarrow A_{\Gamma'}$ be a homomorphism between XL-type free-of-infinity Artin groups. If $f$ is surjective, then $|V(\Gamma)|\geq |V(\Gamma')|$. If $f$ is injective, then $|V(\Gamma)|\leq |V(\Gamma')|$.
\end{corollary}

Note that in the injective case, we obtain the following stronger corollary: 

\begin{corollary}\label{cor:injective_labels}
    Let $f:A_\Gamma \hookrightarrow A_{\Gamma'}$ be an injective homomorphism between XL-type free-of-infinity Artin groups of rank at least $3$. Then $\Gamma$ is isomorphic (as a labelled graph) to a subgraph of $\Gamma'$.
\end{corollary}

\paragraph{XXXL-type Artin groups.} We are also able to study Artin groups over non-complete graphs, specifically in the case where the labels all satisfy $m_{st} \geq 6$.
Our results will involve the following graph-theoretic notions: given a connected simplicial graph $\Gamma$, a vertex $v$ of $\Gamma$ is a \textbf{cut-vertex} if $\Gamma - v$ is disconnected.
An edge $e$ of $\Gamma$ is a \textbf{separating edge} if there exist two induced connected subgraphs $\Gamma_1, \Gamma_2$ such that $\Gamma = \Gamma_1 \cup \Gamma_2$ and $\Gamma_1 \cap \Gamma_2 = e$.
Note that this definition is coherent with the literature on Artin groups \cite{crisp2005artin, an2022automorphism} but differs from another notion of separating edge (sometimes also referred to as ``bridge'') found in the graph theory literature.

For this family of Artin groups, we completely characterise when the group is co-hopfian:

\begin{theorem}\label{thm:XXXL_cohopf}
     An Artin group of XXXL-type $A_\Gamma$ is co-hopfian if and only if $\Gamma$ is connected, is not an edge, and does not have a cut-vertex.
\end{theorem}

We are also able to describe their automorphism group. We provide a generating set, that generalises that of Crisp's (\cite{crisp2005automorphisms}), for the automorphism in the XXXL-case, as soon as the underlying graph does not have a cut-vertex. 

\begin{theorem}
    Let $A_\Gamma$ be an XXXL-type Artin group, such that $\Gamma$ is connected, is not an edge, and does not have a cut-vertex. Then $\Aut(A_\Gamma)$ is finitely generated. (See Theorem~\ref{thm:Aut_generating_set} for an explicit generating set.) 
\end{theorem}

As a corollary, we characterise precisely when the outer automorphism group is finite: 

\begin{corollary}\label{cor:XXXL_split}
    Let $A_\Gamma$ be an Artin group of XXXL-type. Then $\Out(A_\Gamma)$ is finite if and only if $\Gamma$ connected, is not an even edge, and does not have a cut-vertex or a separating edge. Moreover, when $\Out(A_{\Gamma})$ is finite, $\Aut(A_{\Gamma})$ is generated by the conjugations, the graph automorphisms, and the global inversion.
\end{corollary}

Using a different approach, Jones recently proved (using results from this article) that the outer automorphism group of every XXXL-type Artin group over a connected graph is finitely presented \cite{OliAut}.

\paragraph{Strategy of the proofs and structure of the paper.}

Our proofs are geometric in nature and rely on the action of large-type Artin groups on their modified Deligne complex, a simplicial complex with a rich CAT(0) geometry. (See Section~\ref{background} for background.) We present below the overall strategy to study injective homomorphisms, for the sake of simplicity:
\medskip

\textit{Step 1:} In Section~\ref{sec:image_generator}, we show that standard generators are sent to powers of conjugates of standard generators or their inverses.
This implies in particular that for every standard generator $v\in V(\Gamma)$, there exists an element $g_v\in A_{\Gamma'}$ and a standard generator $w \in V(\Gamma')$ such that $f(v) = g_v w^{\pm 1} g_v^{-1}$.
This step relies on a characterisation of the elements that have a centraliser that contains a non-abelian free group (see Proposition~\ref{PropCentralisers}).
\medskip 

We then want to show that if $\Gamma$ has no cut-vertex and no separating edge, then the elements $g_v$ from Step 1 can be taken to be the same element $g \in A_{\Gamma'}$.
The case of a complete graph had already been treated by Vaskou \cite{vaskou2023automorphisms}, which allows us to prove Theorem \ref{thm:structure_hom} in Section~\ref{sec:free_of_infinity}.
 
From a geometric perspective, elements of the form $g_v w^{\pm 1} g_v^{-1}$ are very well understood: their fixed-point set in the Deligne complex of $A_{\Gamma'}$ is a so-called \textbf{standard tree} (see Lemma~\ref{lem:standardtrees}), and one can thus try to understand the set of images of the standard generators $v \in V(\Gamma)$ by controlling the combinatorics of the standard trees associated to the $f(v)$'s, and their intersections. 
\medskip 

\textit{Step 2:} In Section~\ref{sec:cycle_trees}, we prove than when $\Gamma$ is a cycle, the elements $g_v$ from Step 1 can be taken to be the same element $g \in A_{\Gamma'}$.
We do this by controlling the combinatorial geometry of the associated ``cycle of standard trees'' in $D_{\Gamma'}$.
Using Gauss--Bonnet-style arguments, we prove that this cycle of standard trees determines a loop of $D_{\Gamma'}$ that is contained in a single translate of the fundamental domain of $D_{\Gamma'}$, which is enough to prove that all the $g_v$'s can be taken to be equal. 

\medskip 

\textit{Step 3:} In Section~\ref{sec:no_separating}, we generalise the above Step to the case of a graph $\Gamma$ with no cut-vertex and no separating edge.
In order to do this, we introduce the graph of induced cycles (Definition~\ref{def:graph_of_cycles}) that encodes the patterns of induced cycles and their intersections.
For such graphs $\Gamma$, this graph of induced cycles has a sufficiently rich combinatorial structure to allow us to prove that all $g_v$'s can be taken to be equal. 

\medskip 

This approach is already enough to cover the case of injective homomorphisms between free-of-infinity XL-type Artin groups.
In the case where $\Gamma$ has no cut-vertex but has separating edges, there is an additional step:

\medskip 

\textit{Step 4:} We generalise the previous step to the case where $\Gamma$ admits separating edges (but no cut-vertex).
The graph $\Gamma$ decomposes as a union of chunks.
The pattern of intersection of these chunks is encoded in a tree.
One can thus use results from previous steps to deal with individual chunks.
Moreover, one can use certain edge-twists isomorphisms to control the passage from one chunk to an adjacent one.
This allows us to prove Theorem~\ref{thm:cohopf_XXXL_characterisation} in Section~\ref{sec:coHopf}, and Theorem~\ref{thm:Aut_generating_set} in Section~\ref{sec:Aut}.

\paragraph{Generalisations and necessity of the hypotheses.}

The results proved in this article for XXXL-type Artin groups hold more generally for XL-type Artin groups, as soon as one can prove the corresponding Cycle of Standard Trees Property~\ref{def:cycle_standard_trees_property} mentioned in Step 2 above for larger classes of Artin groups. We thus ask:

\begin{question}
    Does the Cycle of Standard Trees Property~\ref{def:cycle_standard_trees_property} hold for large-type Artin groups?
\end{question}

We now give some examples to illustrate why some of the hypotheses that appear in Theorem~\ref{thm:structure_hom} cannot be removed.

\begin{example}
Theorem~\ref{thm:structure_hom} does not extend to general two-dimensional Artin groups, as illustrated with the following example:  
Let $A_{\Gamma} = \langle a, b, c \rangle$ and $A_{\Gamma'} = \langle s, t, r \rangle$ be two Artin groups of rank $3$, such that
\[m_{ab} = 3, \ m_{ac} = 7, \ m_{bc} = 7, \text{ and } m_{st} = 3, \ m_{sr} = 2, \ m_{tr} = 7.\]
We can define a homomorphism $f : A_{\Gamma} \rightarrow A_{\Gamma'}$ given by
\[f(a) = s, \ f(b) = t, \ f(c) = trt^{-1}\]
that is not of the form described in Theorem~\ref{thm:structure_hom}.
The failure here comes from the fact that the three standard trees for $f(a), f(b), f(c)$ form a triangle of standard trees that does not satisfy the Cycle of Standard Trees Property, as illustrated in the picture below.  (In particular, the Cycle of Standard Trees Property does not hold for more general two-dimensional Artin groups.)

\begin{figure}[H]
\centering
\includegraphics[scale=1]{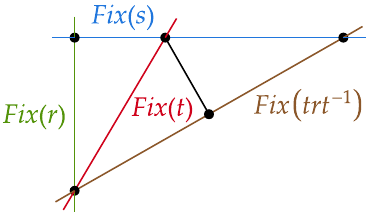}
\caption{The triangle formed by the standard trees $\Fix(s)$, $\Fix(t)$ and $\Fix(trt^{-1})$ is not contained in a translate of the fundamental domain.}
\label{2-dimensional_fail}
\end{figure}
\end{example}

\begin{example}
    For a different reason, general large-type Artin groups cannot be covered either.
    Let $A_{\Gamma} = \langle a, b, c \rangle$ and $A_{\Gamma'} = \langle s, t, r \rangle$ be two Artin groups of rank $3$, such that
    \[m_{ab} = 4, \ m_{ac} = 4, \ m_{bc} = 4, \text{ and } m_{st} = 4, \ m_{sr} = 3, \ m_{tr} = 3.\]
    Now let $f:A_\Gamma \to A_{\Gamma'}$ defined by
    \[f(a) = r^2, \ f(b) = s, \ f(c) = t\]
    This map is well defined, as $\Pi (r^2,s;4) = \Pi (s,r^2;4)$ and $\Pi (r^2,t;4) = \Pi (t,r^2;4)$ in $A_{\Gamma´}$.
    The homomorphism $f$ does not fit the structure in Theorem \ref{thm:structure_hom}, and the failure comes from the failure of Lemma \ref{lem:alternating product equality} for dihedral Artin groups with label $3$.
\end{example}

\begin{example}
    The hypothesis that the image is not contained in virtual product cannot be removed either. Consider the dihedral Artin groups $A_4$ and $A_3$ given by the non-standard presentations $\langle a,b \ | \ b^2=a^{-1}b^2a \rangle$ and $\langle c,d \ | \ c^2=d^3 \rangle$ respectively.
    Then there is an injective morphism $f:A_4 \hookrightarrow A_3$ given by 
    \[f(a)=dcd, \ f(b)=c\]
    which is not of the form given by Theorem~\ref{thm:structure_hom}.
\end{example}

\paragraph{Acknowledgements.} The second author was partially supported by the EPSRC Grant EP/S010963/1. The third author is supported by the Postdoc Mobility $\sharp$P500PT$\_$210985 of the Swiss National Science Foundation. We also thank Luis Paris for pointing out that Lemma~\ref{lem:alternating product equality} does not hold for the dihedral Artin group $A_{st}$ when $m_{st}=3$ (as was stated in a previous version). 

\section{Background on Artin groups}\label{background}

In this section, we recall some standard results about Artin groups that will be used throughout this article.

\subsection{Artin groups and their Deligne complexes}

Given a subset $S\subseteq V(\Gamma)$, the subgroup $\langle S \rangle \leqslant A_\Gamma$ is called a \textbf{standard parabolic subgroup} of $A_\Gamma$. Conjugates of standard parabolic subgroups are called \textbf{parabolic subgroups}.
We recall following fundamental result: 

\begin{theorem}[van der Lek \cite{van1983homotopy}]\label{thm:vdL}
    For a subset $S\subseteq V(\Gamma)$, the parabolic subgroup $\langle S \rangle$ is isomorphic to the Artin group $A_{\Gamma'}$, where $\Gamma'$ is the induced subgraph of $\Gamma$ with vertex set~$S$.
Moreover, given two subsets $S, S' \subseteq V(\Gamma)$, we have $\langle S\rangle  \cap \langle S'\rangle = \langle S\cap S' \rangle$.
\end{theorem}

Because of this result, we will denote by $A_{\Gamma'}$ a standard parabolic subgroups of $A_\Gamma$, for $\Gamma'$ an induced subgraph of $\Gamma$.

An Artin group $A_\Gamma$ is said to be of \textbf{spherical type} if the corresponding Coxeter group $W_\Gamma$ is finite. A particular example is the case of a \textbf{dihedral} Artin group, which is an Artin group with a single edge as presentation graph. Spherical parabolic subgroups play a crucial role in our current understanding of Artin groups:

\begin{definition}
    The \textbf{(modified) Deligne complex} of an Artin group $A_\Gamma$ is the simplicial complex $D_\Gamma$ defined as follows: \begin{itemize}
    \item Vertices correspond to left cosets of standard parabolic subgroups of spherical type.
    \item For every $g \in A_\Gamma$ and for every chain of induced subgraphs $\Gamma_0 \subsetneq \cdots \subsetneq \Gamma_k$ such that $A_{\Gamma_0}, \ldots, A_{\Gamma_k}$ are of spherical type, we put a $k$-simplex between the vertices $gA_{\Gamma_0}, \ldots, gA_{\Gamma_k}$.
    \end{itemize}
    Equivalently, the (modified) Deligne complex $D_\Gamma$ is the geometric realisation of the poset of left cosets of standard parabolic subgroups of spherical type.
    
    The group $A_\Gamma$ acts on $D_\Gamma$ by left multiplication on left cosets, and we denote by $K_{\Gamma}$ the fundamental domain of this action, that is, the subcomplex induced by the vertices of the form $1 \cdot A_{\Gamma'}$.
\end{definition}

\begin{definition}
    A vertex $g A_{\Gamma'}$ of the Deligne complex will be said to be of \textbf{type} $n$ if $|V(\Gamma')| = n$.
\end{definition}

It is a standard result (for instance, see \cite{charney1995k}) that when $A_{\Gamma}$ is large-type, then every vertex of $D_{\Gamma}$ has type $\leq 2$. In particular, $D_{\Gamma}$ is a two-dimensional simplicial complex.

\subsection{Standard trees and the coned-off Deligne complex}

Throughout this section we suppose that $A_{\Gamma}$ is an Artin group of large-type.

Recall that the \textbf{fixed-set} of an element $g \in A_{\Gamma}$ on the Deligne complex $D_{\Gamma}$ is the (possibly empty) set
\[\Fix(g) \coloneqq \{x \in D_{\Gamma} \ | \ g \cdot x = x \}.\]

We have a complete description of the possible fixed-point sets:

\begin{lemma} \cite[Lemma 8]{crisp2005automorphisms}\label{lem:standardtrees}
    Let $g \in A_{\Gamma} \backslash \{1\}$. Then exactly one of the following happens:
    \begin{enumerate}
        \item If $g$ is conjugated to a power of a standard generator, then $g \in h \langle a \rangle h^{-1}$ for some $a \in V(\Gamma)$ and $h \in A_{\Gamma}$. In particular, $\Fix(g)$ is a simplicial tree. 
        \item If $g$ is contained in a dihedral Artin parabolic subgroup of $A_{\Gamma}$, then $g \in h \langle a, b \rangle h^{-1}$ for some $a, b \in V(\Gamma)$ and $h \in A_{\Gamma}$. Moreover if $g$ does not satisfy (1), then $\Fix(g) = h A_{e^{ab}}$ is a single type $2$ vertex, where $e^{ab}$ denote the edge of $\Gamma$ spanned by $a$ and $b$.
        \item Otherwise $\Fix(g)$ is trivial.
    \end{enumerate}
\end{lemma}

\begin{definition}[{\cite{martin2022acylindrical}}]
    A \textbf{standard tree} of $D_\Gamma$ is the fixed-point set of a conjugate of a standard generator of $A_\Gamma$.  
\end{definition}

We have the following useful properties of standard trees: 

\begin{lemma} \cite[Corollary 2.18, Corollary 2.16]{hagen2024extra} \label{lem:standard_trees_properties}
    Two distinct standard trees of $D_\Gamma$ are either disjoint or intersect along a unique vertex of type $2$ of $D_\Gamma$. 

    Moreover, let $x$ be a conjugate of a standard generator. Then for every non-zero integer $k$, we have $\Fix(x^k) = \Fix(x)$. \qedhere
\end{lemma}

In \cite{martin2022acylindrical}, Martin--Przytycki introduced the \textbf{coned-off Deligne complex}, denoted $\widehat{D}_\Gamma$, obtained from $D_\Gamma$ by coning-off every standard tree of $D_\Gamma$. We have the following: 

\begin{theorem}[\cite{martin2022acylindrical}]\label{thm:cone-off}
    Let $A_\Gamma$ be a large-type Artin group of hyperbolic type. Then the coned-off Deligne complex $\widehat{D}_\Gamma$ admits a piecewise hyperbolic CAT($-1$) metric, with an action of $A_\Gamma$ on it by isometries. Moreover, the action of $A_\Gamma$ on $\widehat{D}_\Gamma$ is acylindrical.
\end{theorem}

The following observation will be useful:

\begin{lemma}\label{lem:abstract_dihedral_subgroups}
    Let $A_\Gamma$ be a large-type Artin group of hyperbolic type. Let $H$ be a subgroup of~$A_\Gamma$ isomorphic to a dihedral Artin group. Then there exists a dihedral parabolic subgroup of $A_\Gamma$ containing $H$.
\end{lemma}

\begin{proof}
    Since $H$ virtually splits as a product of the form $\mathbb{Z}\times F_k$ and acts acylindrically on $\widehat{D}_\Gamma$, it follows from \cite{osin2013} that $H$ has bounded orbits.
    Since $\widehat{D}_\Gamma$ is CAT($-1$) by Theorem~\ref{thm:cone-off} and the action is without inversion, $H$ fixes a vertex of $\widehat{D}_\Gamma$.
    Thus, either $H$ is contained in a (necessarily dihedral) parabolic subgroup, or $H$ preserves a standard tree.
    The latter is impossible since the stabiliser of a standard tree is of the form $\mathbb{Z}\times F_k$ by \cite[Lemma 4.5]{martin2022acylindrical}, and such subgroups cannot contain a dihedral Artin group. 
\end{proof}

\subsection{Centralisers of elements}

Centralisers of elements will play a key role in understanding the image of a standard generator under a homomorphism. Recall that centralisers in the large-type case are fully classified. In particular, we have the following:

\begin{proposition} \cite[Remark 3.6]{martin2023characterising} \label{PropCentralisers}
    Let $A_{\Gamma}$ be a large-type Artin group of hyperbolic type of rank at least $3$, with $\Gamma$ connected. Then the centraliser $C(g)$ of a non-trivial element $g \in A_{\Gamma}$ contains a subgroup isomorphic to $\mathbb{Z} \times F$ where $F$ is a non-abelian free group, if and only if up to conjugation, one of the following happens:
    \begin{enumerate}
        \item $g$ is conjugated to a power of a standard generator $s \in V(\Gamma)$, where $s$ is not the tip of an even-labelled leaf of $\Gamma$, or
        
        \item $g$ belongs to the centre of a dihedral Artin parabolic subgroup of $A_{\Gamma}$.
    \end{enumerate}
\end{proposition}

\subsection{Dihedral Artin groups}

Artin groups on two generators play a special role, as they appear as vertex stabilisers in the Deligne complex of a large-type Artin group. We recall here a few results about these groups, and prove certain lemmas that will be used throughout this article. For an integer $m_{st} \geq 3$, we denote by $A_{st}$ the \textbf{dihedral} Artin group on the standard generators $s, t$. 

We consider the usual quotient map $\pi : F_{st} \twoheadrightarrow A_{st}$. A \textbf{simple element} is a word $w \in F_{st}$ that is a non-empty strict subword of either $\Pi(s,t;m_{st})$ or $\Pi(t,s;m_{st})$. Let
\[\Delta_{st} \coloneqq \pi(\Pi(s,t;m_{st})) = \pi(\Pi(t,s;m_{st})).\]
be the \textbf{Garside element} of $A_{st}$. It is known that the centre of $A_{st}$ is generated by $z_{st} \coloneqq \Delta_{st}$ if $m_{st}$ is even, and by $z_{st} \coloneqq \Delta_{st}^2$ if $m_{st}$ is odd (\cite{brieskorn1972artin}).
    
It is a standard result that for every word $g \in F_{st}$ there is a word $Gars(g) \in F_{st}$ representing the same element of $A_{st}$ as $g$, called the \textbf{Garside normal form} of $g$, such that one can write
\[Gars(g) = u_1 \cdots u_k \cdot \Pi(s,t;m_{st})^{N_k},\]
where the $u_i$'s are simple elements that have the property that the last letter of $u_i$ always coincides with the first letter of $u_{i+1}$, and $N_k \in \mathbb{Z}$ (with the convention that $\Pi(s,t;m_{st})^{0}$ is the empty word, and that $\Pi(s,t;m_{st})^{N_i} = \Pi(t^{-1},s^{-1};m_{st})^{|N_i|}$ if $N_i<0$). 
Moreover, the ordered sequence of simple elements $u_1, \ldots, u_k$
associated with $g$ is unique.
An algorithm for constructing this Garside normal form was given in \cite{mairesse2006growth}, and is explained in detail in \cite[Algorithm 4.15]{vaskou2021acylindrical}.
When the word  $g$ has no subword equal to $\Pi(s,t;m_{st})$,  $\Pi(t,s;m_{st})$, $\Pi(s^{-1},t^{-1};m_{st})$ or  $\Pi(t^{-1},s^{-1};m_{st})$, this algorithm admits a simpler form, which we present below:

\begin{algo}[Simplified algorithm for Garside normal form {\cite[Algorithm 4.15]{vaskou2021acylindrical}}]\label{Garside_algo}
We introduce the following notations: for a word $w$, we define $\widetilde{w}$ to be the same word as $w$ if $m_{st}$ is even, and to be the word obtained from $w$ by swapping all the $s$'s and the $t$'s if $m_{st}$ is odd. Note that the words $w \cdot \Pi(s,t;m_{st})$ and $\Pi(s,t;m_{st}) \cdot \widetilde{w}$ represent the same element of $A_{st}$, as $\Delta_{st}$ is central in $A_{st}$ when $m_{st}$ is even, and $\Delta_{st}$ conjugates $s$ into $t$ in $A_{st}$ when $m_{st}$ is odd. 

Let $g$ be a word that does not contain a subword of the form $\Pi(s,t;m_{st})$,  $\Pi(t,s;m_{st})$, $\Pi(s^{-1},t^{-1};m_{st})$ or  $\Pi(t^{-1},s^{-1};m_{st})$.
We construct by induction a sequence of words $g_i$, $i \geq 0$, representing the same element of $A_{st}$ as $g$, and with decompositions of the form
\[g_i = (u_1\cdots u_i) \cdot v_i \cdot \Pi(s,t;m_{st})^{N_i}\]
where the $u_1, \ldots, u_i$ are simple elements, $v_i$ is either a subword of $g$ or of $\widetilde{g}$, and $N_i\in \mathbb{Z}$ (with the same convention as above for powers of $\Pi(s,t;m_{st})$).
For the initialisation stage, we set $v_0 = g$ and $N_0 = \varnothing$.
The word $g_{i+1}$ is obtained from $g_i$ by performing the following: 
\begin{itemize}
    \item If the first letter of $v_i$ is positive, let $u_{i+1}$ be the maximal simple element that appears as a left prefix of $v_i$.  We then define $v_{i+1}$ as the unique subword of $v_i$ such that $v_i = u_{i+1} \cdot v_{i+1}$, and we set $N_{i+1} = N_i$.
    \item If the first letter of $v_i$ has negative exponent, let $n_{i+1}$ be the maximal strict subword of $\Pi(s^{-1},t^{-1};m_{st})$ or $\Pi(t^{-1},s^{-1};m_{st})$ appearing as a left prefix $v_i$, and let $z_{i+1}$ be the unique subword of $v_i$ such that $v_i = n_{i+1} \cdot z_{i+1}$. Let $m_{i+1}$ be the unique simple element such that the word $n_{i+1}$ represents the same element of $A_{st}$ as $\Pi(s^{-1},t^{-1};m_{st}) \cdot m_{i+1}$ or $\Pi(t^{-1},s^{-1};m_{st})\cdot m_{i+1}$. We thus have that $g_i$ represents the same element of $A_{st}$ as 
    $$g_{i+1} = (u_1\cdots u_i \cdot \widetilde{m}_{i+1}) \cdot \widetilde{z}_{i+1} \cdot (\Pi(s^{-1},t^{-1};m_{st})\cdot \Pi(s,t;m_{st})^{N_i})$$
    and so we set $u_{i+1} = \widetilde{m}_{i+1}$, $v_{i+1}= \widetilde{z}_{i+1}$, and $N_{i+1} = N_i -1$.
\end{itemize}
Note that since the length of $v_i$ decreases at each step,  the word $v_i$ becomes empty after finitely many steps, at which point the algorithm terminates.
The resulting decomposition 
$$ g_k = u_1 \cdots u_k \cdot \Pi(s,t;m_{st})^{N_k}$$ is a Garside form for $g$ by \cite[Algorithm 4.15]{vaskou2021acylindrical}. 
\end{algo}

\begin{example}
    In the dihedral Artin group $A_{st}$ with $m_{st}=5$, the above algorithm applied to the word  $g_0= st^2s^{-1}t^{-2}$ yields the sequence of words
    $$g_1 = st \cdot (ts^{-1}t^{-2}), ~~ g_2 = st\cdot t \cdot (s^{-1}t^{-2}), ~~g_3 = st\cdot t \cdot tst \cdot (t^{-1})\cdot \Pi(s,t;5)^{-1},$$ and finally the Garside normal form $g_4 = st \cdot t \cdot tst \cdot tsts \cdot \Pi(s,t;5)^{-2}$.
\end{example}

As an application of Garside normal forms, we mention the following lemma, which will be used several times in this article: 

\begin{lemma}\label{lem:alternating product equality}
    Let $A_{st}$ be a dihedral Artin group with coefficient $m_{st} \geq 4$ and standard generators $s, t$.
    Let $ m, \ell$ be non-zero integers, and let $k>1$ be a positive integer.
    The equality 
    \[\Pi(s^m,t^\ell;k) = \Pi(t^\ell,s^m;k)\]
    holds in $A_{st}$ if and only if $m=\ell=\pm 1$, and $k$ is a multiple of $m_{st}$.
\end{lemma}

\begin{proof}
    The condition $m=\ell=\pm 1$ is clearly sufficient, so let us check that it is necessary.
    Without loss of generality (up to taking inverses), we can assume that $m$ is positive. 
    Note that since $m_{st}\geq 4$, if the condition $m=\ell=\pm 1$ does not hold, then none of the words in the statement contains a subword equal to the Garside element or its inverse.
    Hence we can follow Algorithm~\ref{Garside_algo} to compute the Garside normal form of these words.
    
    Let us first show that $\ell$ is positive.
    It will be enough to compare the first simple element in the Garside normal forms.
    If $\ell$ is negative, one checks that the first simple element in the Garside normal form of $\Pi(s^m,t^\ell;k)$ is $s$, while the first simple element in the Garside normal form of $\Pi(t^\ell, s^m;k)$ is $\Pi(s,t;m_{st}-1)$.
    By uniqueness of the Garside normal form, the corresponding elements are distinct.
    
    Let us now show that $m=\ell=1$. By symmetry of the equation, it is enough to show that $m=1$. 
    Again, it will be enough to compare the first simple element in the Garside normal forms.
    If $m>1$, one checks that the first simple element of $\Pi(s^m,t^\ell;k)$ is $s$, while the first simple element of $\Pi(t^\ell, s^m;k)$ is either $t$ or $ts$ (depending on whether $\ell>1$ or $\ell=1$, respectively).
    By uniqueness of the Garside normal form, the corresponding elements are distinct.
    
    Finally, assume that $\ell = m =1$, and let us show that $k$ is a multiple of $m_{st}$. If that were not the case, then let us write the Euclidean division $k = qm_{st} + r$ with $0<r<m_{st}$.
    Then one checks that $\Pi(s,t;k)$ and $\Pi(t, s;k)$ have Garside normal forms  $\Pi(s,t;r)\cdot\Pi(s,t;m_{st})^q$ and $\Pi(t, s;r)\cdot\Pi(s,t;m_{st})^q$
    respectively.
    As these are distinct, the corresponding elements are distinct by uniqueness of the Garside normal form, a contradiction. 
\end{proof}

\begin{remark}
    If $m_{st}=3$, the above statement no longer holds, as we have  relations such as $\Pi(s^2, t;4) = \Pi(t, s^2;4)$. However, we still get in that case that $k=\pm 1$ or $\ell = \pm 1$, for otherwise $\Pi(s^m, t^\ell;k)$ does not contain any subword of the form $sts, tst, s^{-1}t^{-1}s^{-1},$ or $t^{-1}s^{-1}t^{-1}$, and the same proof as above applies.
\end{remark}

There is a geometric interpretation of the uniqueness of Garside normal forms, which we now present.
Consider the Cayley graph $\mathrm{Cay}(A_{st}, R)$, where $R$ denotes the set of simple elements of $A_{st}$. Conjugation by the Garside element $\Delta_{st}$ either fixes the standard generators (if $m_{st}$ is even) or swaps them (if $m_{st}$ is odd), so it preserves $R$.
Thus, there is an action of $\langle \Delta_{st}\rangle$ on this graph by multiplication on the \textit{right}.
We denote by $X_n$ the quotient graph, where $n:=m_{st}$.
The following was already observed in \cite{Bestvina1999}, see also \cite{martin2024tits} for a more detailed discussion:

\begin{lemma}\label{lem:Bestvina-tree}
    The graph $X_n$ is a tree of simplices, where the maximal simplices are the translates of the simplex spanned by $1, s, st, \cdots, \Pi(s,t;n-1)$ and of the simplex spanned by $1, t, ts, \cdots, \Pi(t,s;n-1)$. 
    
    Let $T_n$ be the corresponding dual tree, whose vertices are the maximal simplices of $X_n$, and whose edges correspond to simplices of $X_n$ with non-empty intersection. Then the graph $T_n$ is a regular $n$-tree. Moreover, the dihedral Artin group $A_n$ with coefficient $n$ acts transitively on the edges of $T_n$. \qedhere
\end{lemma}

It is straightforward to check that the standard generators $s$ and $t$ act by translations of length $1$ on the tree $T_n$. Moreover, $\mathrm{Axis}(s)\cap \mathrm{Axis}(t)$ is a single edge (see Figure \ref{FigureBestvinaTree}). 

\begin{figure}[H]
\centering
\includegraphics[scale=0.78]{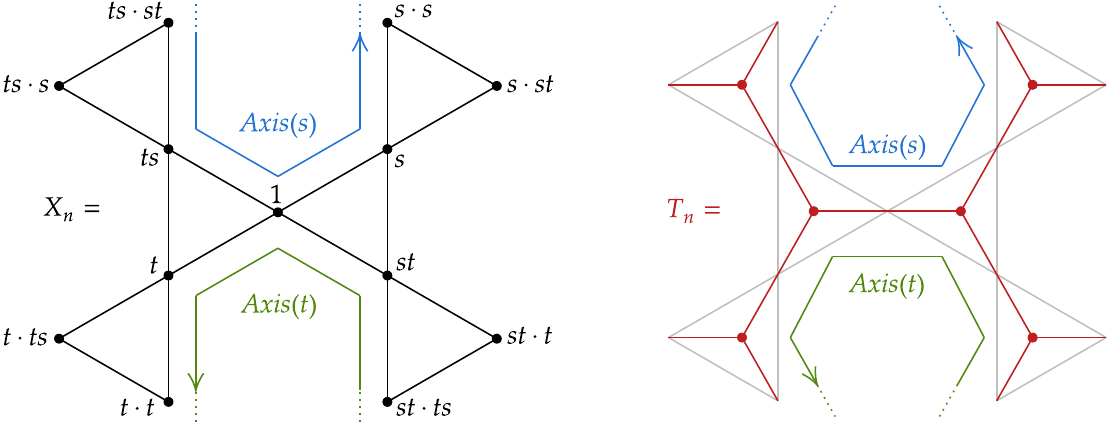}
\caption{The graph $X_n$ and its dual tree $T_n$ when $n = 3$.}
\label{FigureBestvinaTree}
\end{figure}

As an example of application, we mention the following result, which will be used several times in this article:

\begin{lemma}\label{lem:subgroup_conjugates_generators}
    Let $A_{st}$ be a dihedral Artin group with $m_{st} \geq 3$. Let $x, y \in A_{st}$ be two distinct elements that are conjugates of standard generators.
    Then either $x, y$ generate a non-abelian free group, or there exists $g\in A_{st}$ such that  $gxg^{-1}$ and $gyg^{-1}$ are standard generators. 

    In particular, the subgroup generated by any two conjugates of standard generators is either $\mathbb{Z}$, $F_2$, or the whole dihedral group $A_{st}$.
\end{lemma}

\begin{proof}
    Consider the action of $A_{st}$ on the Bestvina tree $T_n$ introduced in Lemma~\ref{lem:Bestvina-tree}. Denote by $e$ the edge $\mathrm{Axis}(s)\cap \mathrm{Axis}(t)$.
    \medskip 

    \noindent \underline{Claim:} Let $z$ be a conjugate of a standard generator. If $\mathrm{Axis}(z)$ contains the edge~$e$, then $z \in \{s, t\}$.
    \medskip 

    Let us write $z = gug^{-1}$ with $u \in \{s, t\}$. We have $\mathrm{Axis}(z) = g\mathrm{Axis}(u)$. Since $\mathrm{Axis}(z)$ contains $e$, it follows that $g$ sends some edge of $\mathrm{Axis}(u)$ to $e$. In particular, $g$ sends some coset of the form $u^*\langle \Delta_{st}\rangle$ to the coset $\langle \Delta_{st}\rangle$. Thus, we have $g = \Delta_{st}^k u^\ell$ for some $k, \ell \in \mathbb{Z}$. In particular, we get 
    $$ z = gug^{-1} = \Delta_{st}^k u\Delta_{st}^{-k}$$
    Since conjugation by $\Delta_{st}$ either fixes the two standard generators (if $m_{st}$ is even) or swaps them (otherwise), the claim follows.

    \medskip 

    Now consider the axes $\mathrm{Axis}(x)$ and $\mathrm{Axis}(y)$.  We consider several cases:
    If $\mathrm{Axis}(x)$ and $\mathrm{Axis}(y)$ meet along a single vertex, then the usual ping-pong argument guarantees that $\langle x, y \rangle \cong F_2$.  
    If $\mathrm{Axis}(x)$ and $\mathrm{Axis}(y)$ are disjoint, then since $x$ and $y$ are infinite order elements with disjoint minsets in $T_n$, a standard result from group actions on trees guarantees that $\langle x, y \rangle \cong F_2$.
    Finally, if $\mathrm{Axis}(x) \cap \mathrm{Axis}(y)$ contains an edge, then in particular it contains a translate of the form $ge$. Thus, the elements $g^{-1}xg$ and $g^{-1}yg$ are conjugates of standard generators whose axes contain $e$. By the previous claim, it follows that $g^{-1}xg \in \{s, t\}$, and similarly $g^{-1}yg \in \{s, t\}$.
\end{proof}

\section{The image of a standard generator}\label{sec:image_generator}

We start our study of the homomorphisms between large-type Artin groups by controlling the possible images of a standard generator.
\medskip

\noindent When the homomorphism under study is injective, we have the following general result:

\begin{lemma}\label{lem:generators_to_generators}
    Let $f : A_{\Gamma} \hookrightarrow A_{\Gamma'}$ be an injective morphism between two XL-type Artin groups of rank at least $3$, with $\Gamma$  connected and with no even-leaves.
    Then for every pair $s, t \in V(\Gamma)$ of adjacent standard generators, there exist elements $g \in A_{\Gamma'}$, $\varepsilon = \pm 1$, and adjacent standard generators $s', t' \in V(\Gamma')$ such that 
    \[f(s) = g(s')^\varepsilon g^{-1}, ~~~ f(t) = g(t')^\varepsilon g^{-1}\]
    In particular, $f$ maps standard generators to conjugates of standard generators or to conjugates of their inverses.
\end{lemma}

\begin{proof}
    By Proposition \ref{PropCentralisers}, for any $s \in V(\Gamma)$, the centraliser $C(s)$ splits as a product $\mathbb{Z} \times F$ where $F$ is a non-abelian free group.
    Since $f$ is injective, the centraliser $C(f(s))$ must contain a product isomorphic to $\mathbb{Z} \times F$ too. Again, by Proposition \ref{PropCentralisers}, this is only possible if $f(s)$ is conjugated to a power of a standard generator, or if $f(s)$ is central in a dihedral Artin parabolic subgroup.
    
    For the sake of contradiction, suppose that there is a standard generator $s \in V(\Gamma)$ such that $f(s)$ is central in a dihedral Artin parabolic subgroup of $A_{\Gamma'}$.
    Up to conjugation, we call it $A_{s't'}$, and we call its central element $z_{s't'}$.
    Because $\Gamma$ is connected and $A_\Gamma$ is not cyclic, $s$ has some neighbour $t \in V(\Gamma)$.
    The subgroup $\langle s, t \rangle$ is a dihedral Artin subgroup of $A_{\Gamma}$, so the image $f(\langle s, t \rangle)$ is isomorphic to a dihedral Artin group, hence is contained in a  dihedral parabolic subgroup by Lemma~\ref{lem:abstract_dihedral_subgroups}.
    In particular, every element of $f(\langle s, t \rangle)$ fixes some vertex of $D_{\Gamma'}$.
    This subgroup contains $f(s) = z_{s't'}^k$ (for some $k \neq 0$), whose fixed-point set is the single vertex with stabiliser $A_{s't'}$.
    It thus follows that $f(\langle s, t \rangle)$ is contained in $A_{s't'}$. 
    But $z_{s't'}^k$ is central in $A_{s't'}$, and since $f(t) \in A_{s't'}$, it follows that $f(\langle s, t \rangle)$ is abelian, a contradiction with $A_{\Gamma}$ being large-type. 
    
    Thus $f(s)$ is conjugated to a power of a standard generator $s' \in V(\Gamma')$.
    Let $s, t \in V(\Gamma)$ be adjacent vertices. Since $f$ is injective, the subgroup $f(\langle s, t \rangle)$ is isomorphic to a dihedral Artin group, thus is contained in the stabiliser of a type 2 vertex.
    In particular, there exist adjacent vertices $s', t' \in V(\Gamma')$ and $g \in \langle s', t' \rangle$ such that \[f(s) = g(s')^n g^{-1}, ~~~ f(t) = g(t')^m g^{-1}\]
    for some $n, m \in \mathbb{Z} \backslash \{0\}$. By Lemma~\ref{lem:alternating product equality}, it follows that the pair $\{f(s), f(t)\}$ is conjugated to the pair $\{s', t'\}$ or the pair $\{(s')^{-1}, (t')^{-1}\}$, which concludes the proof.  
\end{proof}

In order to study more general homomorphisms, we will restrict ourselves to the case where the presentation graphs $\Gamma, \Gamma'$ are complete graphs.
The main result of this section is the following:

\begin{proposition}\label{prop:image_generator}
Let $f: A_\Gamma \rightarrow A_{\Gamma'}$ be a homomorphism as in the statement of Theorem~\ref{thm:structure_hom}. Then for every $s\in V(\Gamma)$, the element $f(s)$ is either trivial or is conjugated to some element of the form $(s')^{\pm 1}$ for some standard generator $s' \in V(\Gamma')$.
\end{proposition}

\noindent In what follows, we denote by $S$ the standard generating set of $V(\Gamma)$. For any $T \subseteq S$, we denote by $\overline{T}$ the set \[\overline{T} \coloneqq T \cup \{z_{st} ~|~ s, t \in T\},\]
where $z_{st}$ was defined in Section 2.4. We denote by $\Gamma_{bar}$ the barycentric subdivision of the graph $\Gamma$, and we see $\overline{S}$ as the vertex set of $\Gamma_{bar}$.

\begin{definition}
    Let $f:A_\Gamma \to A_{\Gamma'}$ be a homomorphism of Artin groups, and let $S_0$ be a minimal subset among all subsets $T \subseteq V(\Gamma)$ such that the restriction $f_T$ of $f$ to $A_T$ is such that $\mathrm{Im}(f_{T}) = \mathrm{Im}(f)$.
    We say that $f$ is \emph{optimal} if  $S_0=V(\Gamma)$.
\end{definition}

\begin{lemma}\label{lem:restriction_injective_S}
    Let $f:A_\Gamma \to A_{\Gamma'}$ be a homomorphism of Artin groups, and let $S_0$ be a minimal subset among all subsets $T \subseteq V(\Gamma)$ such that $\mathrm{Im}(f_{T}) = \mathrm{Im}(f)$.
    Then $S_0$ also satisfies the following: 
    \begin{itemize}
     \item $f$ is injective on $S_0$, and 
     \item if $A_{\Gamma}$ is torsion-free, no element of $\overline{S}_0$ is in the kernel of $f$.    
    \end{itemize}
\end{lemma}

\begin{proof} 
    If $s\in S_0$ is such that $f(s)=1$, then $S_0 \backslash \{s\}$  satisfies $\mathrm{Im}(f_{S_0 \backslash \{s\}}) = \mathrm{Im}(f_{S_0})$, contradicting the minimality of $S_0$.

    If $s, t\in S_0$ are such that $f(s)=f(t)$, then $S_0 \backslash \{s\}$  satisfies $\mathrm{Im}(f_{S_0 \backslash \{s\}}) = \mathrm{Im}(f_{S_0})$, contradicting the minimality of $S_0$.

    If $s, t\in S_0$ are such that $f(z_{st}) = 1$, then some power of $f(s)f(t)$ is trivial. Since $A_{\Gamma'}$ is torsion-free, we have $f(s)=f(t)^{-1}$, so  $S_0 \backslash \{s\}$  satisfies $\mathrm{Im}(f_{S_0 \backslash \{s\}}) = \mathrm{Im}(f_{S})$, contradicting the minimality of $S_0$.
\end{proof}

 Note that to prove Theorem~\ref{thm:structure_hom} for $f$, it is enough to prove it for $f_{S_0}$. We will thus restrict ourselves to the study of optimal homomorphisms.

\begin{lemma}\label{lem:not_cyclic_centraliser}
    Let $f: A_\Gamma \rightarrow A_{\Gamma'}$ be an optimal  homomorphism as in Theorem~\ref{thm:structure_hom}. Then for every $s \in \overline{S}$, there exists $n\geq 1$ such that the centraliser $C(f(s)^n)$ is not cyclic.
\end{lemma}

\begin{proof}
    Consider the barycentric subdivision $\Gamma_{bar}$ of the presentation graph $\Gamma$. We construct a new graph, denoted $f(\Gamma_{bar})$, which is obtained from $\Gamma_{bar}$ by contracting an edge of $\Gamma_{bar}$ between $s$ and $t$ whenever $f(s)$ and $f(t)$ have a common power in $A_{\Gamma'}$.
    Note that $f(\Gamma_{bar})$ is connected, since $\Gamma_{bar}$ is connected and $f(\Gamma_{bar})$ is obtained from it by a sequence of edge contractions.
    
    Suppose by contradiction that $f(\Gamma_{bar})$ is a single vertex. Then for every $s\in S$ there exists an integer $\ell_s$ such that for every $s, t \in S$, we have  $f(s)^{\ell_s} = f(t)^{\ell_t}$. In particular, there exists an element $r$ that is a power of every $f(s)$ for $s\in S$. 
    Since the $\{f(s)\}_{s\in S}$ generate $\im(f)$ by assumption, we would have that $r$ is central in $\im(f)$.
    But $\im(f)$ acts acylindrically on the coned-off Deligne complex $D_{\Gamma'}$ by \cite[Theorem A]{martin2022acylindrical}, and $\im(f)$ is neither cyclic nor contained in a virtual direct product subgroup, so it follows that $\im(f)$ is acylindrically hyperbolic.
    As it is also torsion-free (since $A_{\Gamma'}$ is $2$-dimensional), it follows from Osin \cite[Corollary 7.3]{osin2013} that $\im(f)$ is centre-less.
    So $r$, and hence each $f(s)$, is trivial, contradicting the fact that $f$ is optimal.
    
    Thus, in $f(\Gamma_{bar})$, every vertex has valence $\geq 1$.
    Let $\pi: \Gamma_{bar} \rightarrow f(\Gamma_{bar})$ be the graph contraction.
    Let $s \in \overline{S}$ and $s_0:= \pi(s)$, and let $t_0\in f(\Gamma_{bar})$ be a neighbour of $s_0$.
    By construction of $f(\Gamma_{bar})$, this means that there exists $s'\in \pi^{-1}(s_0)$ and $t\in \pi^{-1}(t_0)$ that are adjacent in $\Gamma_{bar}$.
    Let $n, m \geq 1$ such that $f(s)^n = f(s')^m$.
    Since $s'$ and $t$ commute by construction, we have that $f(t)$ and $f(s)^n = f(s')^m$ commute, hence $\langle f(s)^n, f(t)\rangle \subseteq C(f(s)^n)$.
    Moreover $\langle f(s)^n, f(t)\rangle$ is not cyclic since $f$ is optimal and $s_0 \neq t_0$ in $f(\Gamma_{bar})$. 
\end{proof}

\begin{lemma} \label{LemmaMin(f(s))}
    Let $f: A_\Gamma \rightarrow A_{\Gamma'}$ be an optimal homomorphism as in Theorem~\ref{thm:structure_hom}. Then for every $s\in \overline{S}$, the minset $\Min(f(s))$ is either a type $2$ vertex, a standard tree, or a line contained in a standard tree. In all cases, we have that for any $n \neq 0$, $\Min(f(s)) = \Min(f(s)^n)$.
\end{lemma}

\begin{proof}
    By Lemma \ref{lem:not_cyclic_centraliser}, choose $n\geq 1$ such that $C(f(s)^n)$ is not cyclic.
    Since $A_\Gamma, A_{\Gamma'}$ are of XL-type, they act acylindrically without inversions on their coned-off Deligne complexes $\hat{D}_\Gamma$, $\hat{D}_{\Gamma'}$ by Martin--Przytycki \cite[Theorem A]{martin2022acylindrical}.
    Moreover, $C(f(s)^n)$ acts acylindrically on $\hat{D}_{\Gamma'}$.
    The centraliser $C(f(s)^n)$ is not cyclic, and it is also not virtually cyclic since $A_{\Gamma'}$ is torsion-free.
    It now follows from Osin \cite[Corollary 6.9]{osin2013} that $f(s)^n$, and hence $f(s)$, acts elliptically on $\hat{D}_{\Gamma'}$.
    Since $\hat{D}_\Gamma$ is CAT(0), it follows that $f(s)$ fixes a vertex.
    Thus, since $f$ is optimal, $\Min(f(s))$ is a vertex of $D_\Gamma$, or $f(s)$ is loxodromic and preserves a standard tree of $D_\Gamma$.

    If $\Min(f(s))$ is a type $2$ vertex or a standard tree, then $\Min(f(s)) = \Min(f(s)^n)$ by \cite[Lemma 8]{crisp2005automorphisms}.
    If $\Min(f(s))$ is a line in a standard tree, then $f(s)$ fixes a cone-point $p$ in $\hat{D}_{\Gamma}$, but no neighbour of $p$ as it acts loxodromically on $D_{\Gamma}$. Consequently, $f(s)^n$ also fixes $p$ but no neighbour of $p$. Because $\hat{D}_{\Gamma}$ is CAT(0), fixed-sets are convex, so the above proves that $f(s)^n$ only fixes $p$ in $\hat{D_{\Gamma}}$. In particular, the minset of $f(s)^n$ in $D_{\Gamma}$ is contained in the standard tree associated with the cone point.
    But $f(s)$ act as a translation along a line in the standard, so $f(s)^n$ must do so too. Finally, $\Min(f(s)^n)$ is just a single line.
\end{proof}

\begin{lemma} \label{LemmaMinsetsIntersect}
        Let $f: A_\Gamma \rightarrow A_{\Gamma'}$ be an optimal  homomorphism as in Theorem~\ref{thm:structure_hom}. Then for every $s, t$ adjacent in $\Gamma_{bar}$, the minsets $\Min(f(s))$ and $ \Min(f(t))$ intersect and lie in a common standard tree. Additionally, if both $\Min(f(s))$ and $ \Min(f(t))$ are a vertex, a line or a tree, then they are equal.
\end{lemma}

\begin{proof}
    Since by Lemma \ref{LemmaMin(f(s))} $\Min(f(s)) = \Min(f(s)^n)$ for all $n \geq 1$, the result is clear if $f(s)$ and $f(t)$ have a common power.
    Otherwise, since $f$ is optimal, $f(s)$ and $f(t)$ generate a $\mathbb{Z}^2$-subgroup denoted $H$.
    Since $A_\Gamma$ acts acylindrically without inversions on the CAT(-1) coned-off Deligne complex $\hat{D}_\Gamma$ by Theorem~\ref{thm:cone-off}, it follows that $f(s)$ and $f(t)$ fix a common vertex of $\widehat{D}_\Gamma$.
    Thus, $f(s)$ and $f(t)$ either fix a common vertex of $D_\Gamma$, in which case we are done, or stabilise a common standard tree $T$ of $D_\Gamma$. In the latter case, we look at the induced action of $\Stab(T) \supset H$ on $T$.
    Let $g$ be a conjugate of a standard generator such that the cyclic subgroup $\langle g \rangle$ is the pointwise stabiliser of $T$, and recall that $g$ is central in $\Stab(T)$.
    The quotient $H/(\langle g \rangle \cap H)$ acts with trivial edge stabilisers, hence acylindrically, on $T$. 
    So if $H/(\langle g \rangle \cap H) \cong \mathbb{Z}^2$, we get that $H$, and hence $f(s), f(t)$ fix a vertex of $T$, and we are done.
    Otherwise, either $f(s)$ or $f(t)$ are torsion in $H/(\langle g \rangle \cap H)$, or they have a common power in $H/(\langle g \rangle \cap H)$. Hence, there exist integers $\ell, m, n\geq 0$ not all zero such that $f(s)^\ell = f(t)^mg^n$. If $\ell=0$ (respectively $m=0$), then $\Min(f(t))=T$ (respectively $\Min(f(s))=T$), and thus the two minsets intersect non-trivially.
    Otherwise, we have \[\Min(f(s)) = \Min(f(s)^\ell) = \Min(f(t)^mg^n) = \Min(f(t)^m)= \Min(f(t))\]
    and thus the minsets intersect non-trivially.
    Note that in all cases, when both $\Min(f(s))$ and $ \Min(f(t))$ are a vertex, a line or a tree, then they are equal.
\end{proof}

\begin{lemma}\label{lem:minset_adjacency}
    Let $f: A_\Gamma \rightarrow A_{\Gamma'}$ be an optimal  homomorphism as in Theorem~\ref{thm:structure_hom}. Let $s, t \in \overline{S}$ adjacent in $\Gamma_{bar}$ and assume that the minsets $\Min(f(s))$ and $\Min(f(t))$ are distinct. Then: 
    \begin{itemize}
        \item If $\Min(f(s))$ is either a vertex or a line contained in a standard tree, then $\Min(f(t))$ is a standard tree containing $\Min(f(s))$.
        \item If $\Min(f(s))$ is a standard tree, then $\Min(f(t))$ is either a vertex contained in $\Min(f(s))$, or a line contained in $\Min(f(s))$.
    \end{itemize}
\end{lemma}

\begin{proof}
    By Lemma \ref{LemmaMinsetsIntersect}, if both $\Min(f(s))$ and $ \Min(f(t))$ are a vertex, a line or a tree, then they are equal.
    In any other case they intersect and lie in a common standard tree.
    Thus, there is only one case to rule out. 

    Assume $\Min(f(s))$ is a vertex and $\Min(f(t))$ is a line.
    This is impossible as $f(s)$, which is elliptic with bounded fixed-point set, cannot commute with $f(t)$, which is hyperbolic (for instance see \cite[Lemma 2.29]{vaskou2023isomorphism}).
\end{proof}

\begin{lemma}\label{lem:minset_tree}
        Let $f: A_\Gamma \rightarrow A_{\Gamma'}$ be an optimal homomorphism as in Theorem~\ref{thm:structure_hom}. Then for every $s\in S$ the minset  $\Min(f(s))$ is a standard tree.  
\end{lemma}

\begin{proof}
Suppose by contradiction that for some $s\in S$, the minset $\Min(f(s))$ is either a vertex or a line. We will first need the following result:

\medskip

\noindent \underline{Claim:} For every $t\neq s$ in $S$, we have that either $\Min(f(t)) = \Min(f(s))$ or $\Min(f(t))$ is a standard tree containing $\Min(f(s))$.

\medskip

Indeed, if that were not the case, then it would follow from Lemma~\ref{lem:minset_adjacency} that $T\coloneqq \Min(f(z_{st}))$ is a standard tree and $\Min(f(t))$ is either a line or a vertex contained in $\Min(f(z_{st}))$. Since $f(z_{st})$ commutes with both $f(s)$ and $f(t)$, it follows that both $f(s)$ and $f(t)$ stabilise the standard tree $T= \Min(f(z_{st}))$. We can thus look at the action of $\langle f(s), f(t) \rangle \subseteq \Stab(T)$ on $T$.

The stabiliser of $T$ splits as a direct product of the form $\mathbb{Z}\times F_k$ where the central $\mathbb{Z}$ factor consists of the elements acting trivially on $T$, by \cite[Remark 4.6]{martin2022acylindrical}.
Thus, we can consider the action of the free group $\Stab(T) / \Fix(T) \cong F_k$ on $T$, and it follows that the images of $f(s)$ and $f(t)$ in $F_k$ still have distinct minsets in $T$.
It follows that these projections generate a free subgroup of $F_k$, hence $f(s)$ and $f(t)$ also generate a free subgroup of $A_\Gamma$.
Thus, $s$ and $t$ generate a free subgroup of $A_{\Gamma}$, contradicting the fact that $m_{st}\neq \infty$.
This proves the claim.

\medskip Suppose first that $\Min(f(s))$ is a line contained in a standard tree.
Since distinct standard trees intersect in at most one vertex by Lemma \ref{lem:standard_trees_properties}, it follows from the above claim that for $t\in S$, we either have $\Min(f(t)) = \Min(f(s))$ or $\Min(f(t))$ is the unique standard tree $T$ containing $\Min(f(s))$.
In any case, we get that each $f(t)$ stabilises $T$, hence $\im(f)$ is contained in $\Stab(T)$.
As this stabiliser splits as a direct product by Martin--Przytycki \cite[Lemma 4.5]{martin2022acylindrical}, this contradicts the fact that the image of $f$ is not contained in a subgroup of $A_{\Gamma'}$ that virtually splits as a direct product. 

Suppose now that $\Min(f(s))$ is a single vertex. 
It now follows from the claim that for every $t\in S$, the minset $\Min(f(t))$ is either the vertex $v:=\Min(f(s))$ or a standard tree containing that vertex.
It thus follows that for every $t\in S$, the element $f(t)$ fixes $v$, hence $f(A_{\Gamma})\subseteq \Stab(v)$, contradicting the fact that the image of $f$ is not contained in a dihedral parabolic subgroup of $A_{\Gamma'}$, as those virtually split as direct products of infinite groups.
\end{proof}

At this point, we know from Lemma~\ref{lem:minset_tree} that the image of every standard generator of $A_\Gamma$ is either trivial or conjugated to a \textit{power} of a standard generator of $A_{\Gamma'}$. We will need the following:

\begin{lemma}\label{lem:minset_tree_intersect}
        Let $f: A_\Gamma \rightarrow A_{\Gamma'}$ be an optimal homomorphism as in Theorem~\ref{thm:structure_hom}. Then for every distinct elements $s, t\in S$ with non-trivial images, the standard trees  $\Min(f(s))$ and $\Min(f(t))$ intersect non-trivially.
\end{lemma}

\begin{proof}
    By Lemma~\ref{lem:minset_tree}, the two minsets $T:=\Min(f(s))$ and $T':=\Min(f(t))$ are standard trees, and we can further assume that they are distinct, otherwise there is nothing to prove. 
    By Lemma~\ref{lem:minset_adjacency} we know that $\Min(f(z_{st}))$ is non-trivial and contained in both $T$ and $T'$, so the result follows.
\end{proof}

We are now finally ready to prove Proposition~\ref{prop:image_generator}

\begin{proof}[Proof of Proposition~\ref{prop:image_generator}.]
Let $S_0$ be as in Lemma \ref{lem:restriction_injective_S} and note that $f_{S_0}$ is optimal. Let now $s\in S_0$. Since the image of $f$ is not cyclic by assumption, we can pick some $t\in S_0$ such that $f(s)$ and $f(t)$ do not belong to the same cyclic group.
By Lemma~\ref{lem:minset_tree}, we can pick elements $s', t'\in S'$, $k, l \in \mathbb{Z} \backslash \{0\}$ and $g, h \in A_{\Gamma'}$ such that 
\[f(s) = g(s')^kg^{-1}, ~~~ f(t) = h(t')^\ell h^{-1}.\]
By construction, the two standard trees $\Min(f(s))$ and $\Min(f(t))$ of $D_{\Gamma'}$ are distinct.
Moreover, they intersect by Lemma~\ref{lem:minset_tree_intersect}.
Thus, up to conjugation we can assume that $f(s)$ and $f(t)$ are conjugates of powers of standard generators in a dihedral Artin group.
Since $\Gamma$ is complete, $s$ and $t$ do not generate a free group, so neither do $f(s)$ and $f(t)$.
By applying Lemma~\ref{lem:subgroup_conjugates_generators} to the pair $\{gs'g^{-1},ht'h^{-1}\}$, it follows that the pair $\{gs'g^{-1},ht'h^{-1}\}$ is conjugated to the pair $\{s', t'\}$.
Thus, the pair $\{f(s), f(t)\}$ is conjugated to the pair $\{(s')^k, (t')^\ell\}$.
Up to conjugation, we can thus assume that $g=h=1$.
Now $s'\neq t'$ since $f$ is optimal.
Since $\Gamma$ is complete, we have the equality $\Pi(s,t;m_{st}) = \Pi(t,s;m_{st})$ in $A_\Gamma$, hence the equality $\Pi((s')^k,(t')^\ell;m_{st}) = \Pi((t')^k,(s')^\ell;m_{st})$ in $A_{\Gamma'}$. It now follows from Lemma~\ref{lem:alternating product equality} that $k=\ell= \pm 1$, which concludes the proof.
\end{proof}

\section{Configurations of standard trees}

In Section \ref{sec:image_generator} we saw that under good enough hypotheses, the image under a homomorphism $f : A_{\Gamma} \rightarrow A_{\Gamma'}$ of a standard generator $s \in V(\Gamma)$ is either trivial or conjugated to a standard generator (up to inversion).
In the latter case, $f(s)$ fixes a standard tree in $D_{\Gamma'}$.
If $f$ is injective or optimal, we also know that for every $t \in V(\Gamma)$ adjacent to $s$, the images $f(s)$ and $f(t)$ have intersecting standard trees.
It is thus natural to want to understand the intersection pattern of families of standard trees.
For instance, the following property was a key ingredient to understand the automorphisms of large-type Artin groups over complete graphs:

\begin{lemma} (Triangle of Standard Trees, \cite[Lemma 3.10]{vaskou2023automorphisms}) \label{lem:triangle_standard_trees}
    Let $A_{\Gamma}$ be a large-type Artin group with Deligne complex $D_{\Gamma}$. Let $T_1, T_2, T_3$ be three standard trees of $ D_{\Gamma}$ such that $T_i \cap T_j \neq \varnothing$ whenever $i\neq j$, but such that $T_1\cap T_2\cap T_3 = \varnothing$. 
    Then there exists a unique element $g \in A_{\Gamma}$ such that for every $i$, the translate $g K_{\Gamma}$ contains an edge of $T_i$.
\end{lemma}

Lemma \ref{lem:triangle_standard_trees} can be used to study Artin groups over complete graphs because complete graphs are a union of $3$-cycles. We want to generalise the above property to larger cycles of standard trees, to be able to generalise results past the free-of-infinity case.

\begin{definition}\label{def:cycle_standard_trees_property}
    A \textbf{cycle of standard trees} is a sequence $T_1, \ldots, T_n$ of distinct standard trees of the Deligne complex $D_\Gamma$ such that for every $i \in \mathbb{Z}/n\mathbb{Z}$, the intersection $T_i \cap T_{i+1}$ is a vertex, and such that: 
    \begin{itemize}
        \item for every $i \in \mathbb{Z}/n\mathbb{Z}$, the generators $x_i$ and $x_{i+1}$ of $\Fix(T_i)$ and $\Fix(T_{i+1})$ respectively generate a dihedral Artin group, 
        \item for every distinct $i, j$ with $j \neq i\pm1$, the generators of $\Fix(T_i)$ and $\Fix(T_j)$ generate a non-abelian free group.
    \end{itemize}
    Given a cycle standard trees $T_1, \ldots, T_n$, we associate a loop $\gamma$ of $D_\Gamma$ obtained by concatenating, for $i \in \mathbb{Z}/n\mathbb{Z}$, the geodesic segments $\gamma_i$ of $T_i$ between the vertices $T_i\cap T_{i-1}$ and $T_i\cap T_{i+1}$.
    
    We say that the Artin group $A_\Gamma$ satisfies the \pmb{$N$}\textbf{-Cycle of Standard Trees Property} if the following holds: let $T_1, \ldots, T_n$ be a cycle of standard trees in $D_\Gamma$ with $n \leq N$, and let $\gamma$ be the associated loop.
    Then there exists a (necessarily unique) element $g\in A_\Gamma$ such that $\gamma$ is contained in the translate $gK_{\Gamma}$  of the fundamental domain of $D_\Gamma$.
    We say that the Artin group $A_\Gamma$ satisfies the \textbf{Cycle of Standard Trees Property} if it satisfies the $N$-Cycle of Standard Trees Property for all $N$.
\end{definition}

\begin{definition}\label{def:graph_of_cycles}
We say that a cycle $\gamma \subseteq \Gamma$ is \textbf{induced} if two vertices of $\gamma$ adjacent in $\Gamma$ are always adjacent in $\gamma$.
We define the \textbf{graph of induced cycles} $\mathcal{G}_\Gamma$ of $\Gamma$ as follows:
    \begin{itemize}
        \item The vertices of $\mathcal{G}_\Gamma$ are the induced simple cycles of $\Gamma$ ;
        \item Two induced cycles $\gamma$ and $\gamma'$ are connected by an edge if and only if $\gamma \cap \gamma'$ contains an edge of $\Gamma$.
    \end{itemize}
\end{definition}

\begin{lemma}\label{lem:graph_cycles_connected}
    If a simplicial graph $\Gamma$ is connected and has no cut-vertex then $\mathcal{G}_\Gamma$ is connected. 
\end{lemma}

\begin{proof}
    Assume that $\Gamma$ is connected and has no cut-vertex.
    Let $\gamma$ and $\gamma'$ be simple induced cycles in $\Gamma$.
    Let $e_1,\ldots,e_n$ be a path with $e_1$ in $\gamma$ and $e_n$ in $\gamma'$.
    For each $i$, since $e_i \cap e_{i+1}$ is not a cut-vertex, we can extend $e_i\cup e_{i+1}$ to a simple cycle $\sigma_i$.
    If $\sigma_i$ is not induced, then by adding some of the shortcut edges we can decompose $\sigma_i$ into a chain of induced simple cycles that form a path in $\mathcal{G}_\Gamma$, with the first one containing $e_i$ and the last one containing $e_{i+1}$.
    Using these cycles we get a path from $\gamma$ to $\gamma'$ in $\mathcal{G}_\Gamma$.
\end{proof}

\begin{lemma}\label{lem:decomposition_in_induced_cycles}
    If $\Gamma$ has no separating edge, then for every edge of $\mathcal{G}_\Gamma$ between induced cycles $\gamma$ and $\gamma'$, there exists a vertex $v$ of $\Gamma$, and a sequence of induced cycles $\gamma_1 \coloneqq \gamma, \ldots, \gamma_n \coloneqq \gamma'$ such that the concatenation of the $\gamma_i$'s for $i \in \mathbb{Z}/n\mathbb{Z}$ defines a non-backtracking loop in $\mathcal{G}_\Gamma$, and for every $i\in \mathbb{Z}/n\mathbb{Z}$, the cycles $\gamma_i$ and $\gamma_{i+1}$ share an edge containing $v$.
\end{lemma}

\begin{proof}
    Let $v$ be a vertex in $\gamma \cap \gamma'$ such that there exist edges $e \subset \gamma \cap \gamma'$, $e_1 \subset \gamma \setminus \gamma'$, and $e_2 \subset \gamma' \setminus \gamma$ containing $v$.
    Since $e$ is not a separating edge, we can extend $e_1 \cup e_2$ to a simple cycle $\sigma$ that does not pass through the other terminal vertex of $e$.
    Take $\sigma$ to be minimal length among such cycles.
    If $\sigma$ is induced, we are done.
    If it is not, since $\sigma$ is of minimal length, then by adding finitely many edges incident to $v$, we can subdivide $\sigma$ to obtain induced cycles $\gamma_2, \ldots, \gamma_{n-1}$ such that $(\gamma_i)_i$ defines a simple path in $\mathcal{G}_\Gamma$ going from $\gamma$ to $\gamma'$ (note that none of the edges we added is equal to $e$, since $\sigma$ does not contain the other terminal vertex of $e$).
    Furthermore, if we set $\gamma_1 \coloneqq \gamma$ and $\gamma_n \coloneqq \gamma'$, then for every $i\in \mathbb{Z}/n\mathbb{Z}$, the cycles $\gamma_i$ and $\gamma_{i+1}$ share an edge containing $v$.    
\end{proof}

\begin{figure}[H]
    \centering
    \includegraphics[scale=0.9]{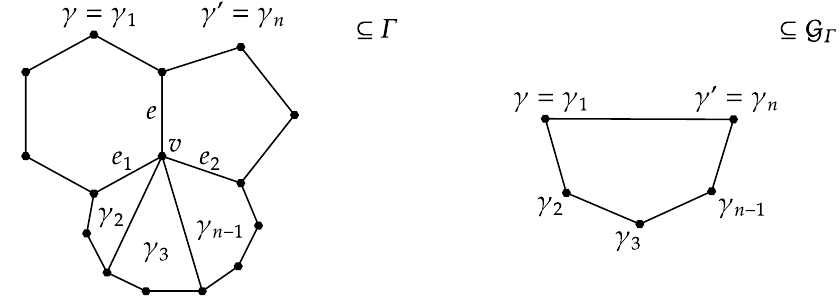}
    \caption{An illustration of the arguments used in the proof of Lemma \ref{lem:decomposition_in_induced_cycles}.}
\end{figure}

\paragraph{The case of injective homomorphisms.}

\begin{lemma}\label{lem:cycles_to_cycles}
       Let $f: A_\Gamma \hookrightarrow A_{\Gamma'}$ be an injective homomorphism between XL-type Artin groups, where $\Gamma$ is a cycle with vertices $v_1, \ldots, v_N$.  

       Then the elements $f(v_i)$ are conjugates of standard generators or their inverses, and the corresponding standard trees $T_1, \ldots, T_N$ form a cycle of standard trees.
\end{lemma}

\begin{proof}
    The first statement directly follows from Lemma \ref{lem:generators_to_generators}.
    For $i \in \mathbb{Z} / n \mathbb{Z}$, the elements $v_i$ and $v_{i+1}$ generate a dihedral Artin group by hypothesis, hence so must the elements $f(v_i)$ and $f(v_{i+1})$ by injectivity of $f$.
    Similarly, when $|i-j| > 1$, the elements $v_i$ and $v_j$ generate a free group hence so must the elements $f(v_i)$ and $f(v_j)$. The result follows.
\end{proof}

\begin{definition}
    Let $f: A_\Gamma \hookrightarrow A_{\Gamma'}$ be an injective homomorphism between XL-type Artin groups, where $\Gamma$ is not an edge and has no cut-vertex or separating edge, and where $A_{\Gamma'}$ satisfies the $N$-Cycle of Standard Trees Property, where $N$ is the maximal length of an induced cycle of $\Gamma$. For every induced cycle $\gamma$ of $\Gamma$ we let $g_\gamma \in A_{\Gamma'}$ be the (unique) element obtained from Definition \ref{def:cycle_standard_trees_property} applied to the restriction $f: A_\gamma \rightarrow A_{\Gamma'}$.
\end{definition}

The crucial lemma is the following:

\begin{lemma}\label{lem:trivial_labels}
   Let $f: A_\Gamma \rightarrow A_{\Gamma'}$ be an injective homomorphism $f:A_{\Gamma} \rightarrow A_{\Gamma'}$ between XL-type Artin groups, where $\Gamma$ is not an edge and has no cut-vertex or separating edge, and where $A_{\Gamma'}$ satisfies the $N$-Cycle of Standard Trees Property, where $N$ is the maximal length of an induced cycle of $\Gamma$. Then for every pair of induced cycles $\gamma, \gamma'$ of $\Gamma$, we have $g_\gamma = g_{\gamma'}$.
\end{lemma}

\begin{proof}
    Let $\gamma$ and $\gamma'$ be two cycles of $\Gamma$ that are adjacent in $\mathcal{G}_\Gamma$.
    Then by Lemma \ref{lem:decomposition_in_induced_cycles} there is a sequence of induced cycles $\gamma_1 \coloneqq \gamma, \ldots, \gamma_n \coloneqq \gamma'$ such that the concatenation of the $\gamma_i$'s for $i \in \mathbb{Z}/n\mathbb{Z}$ defines a non-backtracking loop in $\mathcal{G}_\Gamma$, and for every $i\in \mathbb{Z}/n\mathbb{Z}$, the cycles $\gamma_i$ and $\gamma_{i+1}$ share an edge containing a vertex $v$ of $\gamma \cap \gamma'$. We will assume that this loop has minimal length among all such possible loops. 
    Let $a_i$ be a vertex adjacent to $v$ in $\gamma_i\cap \gamma_{i+1}$ (indices modulo $n$).
    Note that the $a_i$ are all distinct by minimality of the cycle. 

    For every $i$, the dihedral Artin group $\langle v, a_i \rangle$ maps isomorphically onto its image by the injectivity of $f$.
    Thus, there exists a unique vertex $w_i$ of $D_{\Gamma'}$ stabilised by $f(\langle v, a_i \rangle)$, and this vertex is precisely the unique intersection between the standard trees $T_v$ and $T_i$ corresponding to $f(v)$ and $f(a_i)$ respectively (these elements are conjugates of powers of standard generators by Lemma~\ref{lem:generators_to_generators}). 
    
    \medskip
    
    \noindent \underline{Claim.} The vertices $w_i$ are all distinct. 

    \medskip 

    If that were not the case, we could find two indices $i\neq j$ such that the dihedral Artin subgroup $\langle a_i, v \rangle$ and $\langle a_j, v \rangle$ both map inside some dihedral parabolic subgroup $H$ of $A_{\Gamma'}$.
    In particular, since central elements of a dihedral Artin group are the only elements with centralisers virtually of the form $\mathbb{Z} \times F$ in $H$, the elements $z_{a_iv}$ and $z_{a_jv}$ are mapped to some power of $z$, where $z$ generates the cyclic centre of $H$.
    In particular, some power of $z_{a_iv}$ and $z_{a_jv}$ have the same image under $f$, contradicting the injectivity of $f$.
    This proves the Claim.

    \medskip 
    
   For every $i$, the vertices $a_i, v, a_{i+1}$ are contained in the cycle $\gamma_{i+1}$, so by applying Lemma \ref{lem:cycles_to_cycles} and since $A_{\Gamma'}$ satisfies the $N$-Cycle of Standard Trees Property, the vertices $w_i$ and $w_{i+1}$ are at simplicial distance $2$ in the tree $T_v$. Since the $w_i$ are all distinct, and any two consecutive $w_i, w_{i+1}$ are at simplicial distance $2$ in the tree $T_v$ (indices mod $n$), it follows that there is a vertex of $T_v$ (necessarily of the form $h\langle x \rangle$ for $h\in A_{\Gamma'}$ and some standard generator $x \in V_{\Gamma'}$) adjacent to all the $w_i$.

   Set $g_i:= g_{\gamma_i}$ and $g_{i+1}:= g_{\gamma_{i+1}}$, and let us now prove that $g_i =g_{i+1}$. Let $e_i$ be edge of $T_i \cap g_i K_{\Gamma'}$ containing $w_i$ and $e_{i+1}$ be the edge of $T_i\cap g_{i+1} K_{\Gamma'}$ containing $w_i$.
   Note that the angle between $T_v \cap g_iK_{\Gamma'}$ and $T_i \cap g_iK_{\Gamma'}$ is $\leq \pi/3$ by construction of the metric on $D_{\Gamma'}$.
   In particular, if $e_i$ and $e_{i+1}$ were not the same edge, $T_i$ would contain two adjacent edges whose angle is $\leq 2\pi/3$, contradicting the convexity of $T_i$: 

  \begin{center}
    \includegraphics[scale=0.9]{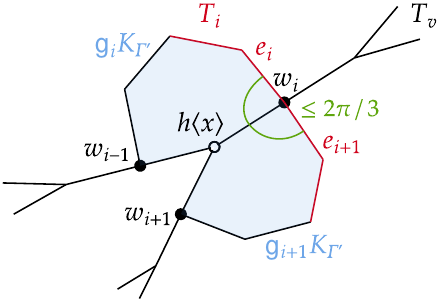}
    \end{center}

    The argument is now the same as in \cite[Figure 5]{vaskou2023automorphisms}: since $g_iK_{\Gamma'}$ and $g_{i+1}K_{\Gamma'}$ share two edges making an angle $\pi/3$, they must share a point with trivial stabiliser by convexity of the fundamental domains, which forces $g_i = g_{i+1}$.
    Since this holds for every $1 \leq i < n$, we obtain $g_\gamma = g_{\gamma'}$.
    The result now follows, since $\mathcal{G}$ is connected by Lemma~\ref{lem:graph_cycles_connected}.
\end{proof}

\paragraph{The free-of-infinity case.}

\begin{lemma}\label{lem:trivial_labels_complete}
   Let $f: A_\Gamma \rightarrow A_{\Gamma'}$ be an homomorphism between XL-type Artin groups, where $\Gamma$ and $\Gamma'$ are complete graphs with at least $3$ vertices. We also assume that the following holds: Every standard generator $v_i$ of $A_\Gamma$ is sent to a conjugate of a standard generator of $A_{\Gamma'}$ or its inverse. Let $T_i$ be the standard tree corresponding to $f(v_i)$ in $D_{\Gamma'}$. Further assume that the $T_i$'s pairwise intersect but every triple intersection $T_i \cap T_j \cap T_k$ is empty. 
   For each triangle $\gamma$ of $\Gamma$ with vertices $v_i, v_j, v_k$, let $g_\gamma$ be the element of $A_{\Gamma'}$ obtained by applying Lemma~\ref{lem:triangle_standard_trees} to $T_i, T_j, T_k$. 
   
   Then for every pair of triangles of $\Gamma$, we have $g_\gamma = g_{\gamma'}$.
\end{lemma}

\begin{proof}
    The proof will be almost identical to that of Lemma~\ref{lem:trivial_labels}. Let $\gamma$ and $\gamma'$ be two triangles of $\Gamma$ that share an edge. Let us denote by $v, a_1$  their common vertices, let $a_2$ be the remaining vertex of $\gamma$, and $a_3$ be the remaining vertex of $\gamma'$. Let also $\gamma''$ be the triangle spanned by $v, a_3, a_1$ (which exists since $\Gamma$ is complete).
    
    It follows from Lemma~\ref{lem:triangle_standard_trees} that for every $i$ the pair $\{f(v), f(a_i)\}$ is conjugated to a pair of standard generators of $A_{\Gamma'}$, hence the dihedral Artin group $\langle v, a_i \rangle$ maps isomorphically onto its image. Thus, there exists a unique vertex $w_i$ of $D_{\Gamma'}$ stabilised by $f(\langle v, a_i \rangle)$, and this vertex is precisely the unique intersection between the standard trees $T_v$ and $T_i$ corresponding to $f(v)$ and $f(a_i)$ respectively. Moreover, the vertices $w_i$ are all distinct by the assumption on empty triple intersections. 
    
    By Lemma~\ref{lem:triangle_standard_trees}, the vertices $w_1, w_2, w_3$ are pairwise at simplicial distance $2$ in the tree $T_v$. Thus, as before there exists a vertex of $T_v$ (necessarily of the form $h\langle x \rangle$ for $h\in A_{\Gamma'}$ and some standard generator $x$ of $A_{\Gamma'}$) adjacent to $w_1, w_2$, and $w_3$. From there, the proof is identical to that of Lemma~\ref{lem:trivial_labels}, and we omit it.
\end{proof}

\section{The case of complete graphs}\label{sec:free_of_infinity}

The goal of this section is to prove Theorem \ref{thm:structure_hom}.

\begin{lemma}\label{lem:config_standard_trees}
    Let $A_{\Gamma}$ be a large-type free-of-infinity Artin group. Let $T_1, \ldots, T_n$ be a collection of standard trees of $D_{\Gamma}$ such that $T_i \cap T_j \neq \varnothing$ whenever $i\neq j$, but such that $T_i \cap T_j \cap T_k = \varnothing$ whenever $i\neq j \neq k$.  Then there exists an element $g \in A_{\Gamma}$ such that the translate $g K_{\Gamma}$ contains an edge of $T_i$ for every $i$.
    In particular, for every $i$ there exists a standard generator $v_i$ of $A_\Gamma$ such that the pointwise stabiliser $\Fix(T_i)$ is generated by $gv_ig^{-1}$.
\end{lemma}

\begin{proof}
    Let $\Lambda$ be the complete graph on $n$ vertices $v_1, \ldots, v_n$. For every $1 \leq i\leq n$, let $x_i$ be the unique generator of the cyclic group of the pointwise stabiliser $\Fix(T_i)$ that is sent to $1$ under the homomorphism $A_\Gamma \rightarrow \mathbb{Z}$ that sends all standard generators to $1$. By Lemma \ref{lem:triangle_standard_trees}, for every $i\neq j$, the elements $x_i, x_j$ generate the dihedral parabolic subgroup $\Fix(T_i \cap T_j)$. Let $m_{ij}$ be the label of that dihedral Artin subgroup (this is well defined up to isomorphism). We label each edge of $\Lambda$ between $v_i$ and $v_j$ with the label $m_{ij}$.
    We can now define a homomorphism $f: A_\Lambda \rightarrow A_\Gamma$ that sends each $v_i$ to $x_i$.

    The result is now a consequence of Lemma~\ref{lem:trivial_labels_complete} applied to $f$.
\end{proof}

We are now almost ready to prove Theorem~\ref{thm:structure_hom} and its consequences.
We first need the following:

\begin{lemma}\label{lem:minset_tree_empty_triple_intersection}
        Let $f: A_\Gamma \rightarrow A_{\Gamma'}$ be an optimal homomorphism as in Theorem~\ref{thm:structure_hom}. Then for every distinct elements $s_1, s_2, s_3 \in S = V(\Gamma)$ such that $T_1:= \Min(f(s_1)), T_2:= \Min(f(s_2)),$ and $T_3:= \Min(f(s_3))$ are distinct standard trees, the intersection $T_1 \cap T_2 \cap T_3$ is empty.
\end{lemma}

\begin{proof}
    We know from Lemma~\ref{lem:minset_tree_intersect} that $T_1 \cap T_2$ is non-empty, hence is a single vertex $v$ by Lemma~\ref{lem:standard_trees_properties}.
    As in the proof of Proposition~\ref{prop:image_generator}, we can assume that $f(s_1)$ and $f(s_2)$ are conjugates of standard generators (or their inverses) in the dihedral Artin group $\Stab(v)$.
    Since $T_1$ and $T_2$ are distinct, the subgroup of $\Stab(v)$ generated by $f(s_1)$ and $f(s_2)$ is not cyclic.
    Since $\Gamma$ is a complete graph, the subgroup of $A_\Gamma$ generated by $s_1, s_2$ is not a free group, so neither is the subgroup of $\Stab(v)$ generated by $f(s_1)$ and $f(s_2)$.
    It now follows from Lemma~\ref{lem:subgroup_conjugates_generators} that $\langle f(s_1), f(s_2) \rangle = \Stab(v)$.
    If by contradiction, we had that the triple intersection $T_1 \cap T_2 \cap T_3$ was non-empty, then $T_3$ would contain $v$, hence $f(s_3)$ would be contained in $\Stab(v) = \langle f(s_1), f(s_2) \rangle$, contradicting the fact that $f$ is optimal. 
\end{proof}

\begin{proof}[Proof of Theorem~\ref{thm:structure_hom}]
    Consider a minimal subset $S_0 \subseteq S$ as in Lemma \ref{lem:restriction_injective_S}. Take a subset $S_1$ of $S_0$ such that the minsets $\Min(f(s))$ for $s \in S_1$ yield a list of all the minsets $\Min(f(s))$ for $s \in S$ without repetition.
    It follows from Lemmas~\ref{lem:minset_tree} and~\ref{lem:minset_tree_intersect} that for every distinct elements $s, t \in S_1$, the minsets $\Min(f(s))$, $\Min(f(t))$ are standard trees that intersect non-trivially, and in particular the intersection is a single vertex.
    Also, by Lemma \ref{lem:minset_tree_empty_triple_intersection} the triple intersections of these trees are empty, so $f$ satisfies the hypotheses of Lemma \ref{lem:config_standard_trees}.
    Hence there exists a translate of the fundamental domain of $D_\Gamma$ that contains edges of each of these standard trees.
    Thus, there exists $g\in A_{\Gamma'}$ a subset $S_1'\subseteq S'$ with $|S_1|=|S_1'|$ and a bijection $\iota: S_1 \rightarrow S_1'$ such that for every $s\in S_1$, $f(s)$ is of the form $g\iota(s)^{n_s}g^{-1}$ for some $\iota(s) \in S_1'$ and some non-zero integer $n_s$.
    In particular, we get that $f(s)\in gA_{S_1'}g^{-1}$ for all $s\in S_1$. By Lemma \ref{lem:alternating product equality}, we either have $n_s=1$ for all $s\in S_1$ or $n_s = -1$ for all $s \in S_1$. In particular, the $f(s)$ generate $gA_{S_1'}g^{-1}$, and it follows that $\mathrm{Im}(f)$ contains $gA_{S_1'}g^{-1}$.

    Let us now show by contradiction that $S_1=S_0$. Indeed, if there existed $t\in S_0 \backslash S_1$, let $s\in S_1$ such that $\Min(f(t))= \Min(f(s)) =: T$. We can then define $S_1' \coloneqq S_1 \backslash (\{s\} \cup \{t\})$. The same reasoning as in the previous paragraph implies that $f(t)$ is a conjugate of a standard generator, i.e. a generator of the pointwise stabiliser of $T$. We thus have $f(t) = f(s)^{\pm 1}$, which contradicts the optimality of $S$. Thus, $S_1 = S_0$.

    Finally, let $s, t \in S_1$ be two distinct standard generators. Since in $A_\Gamma$, we have $\Pi(s, t; m_{st}) = \Pi(t, s; m_{st}) $, by applying the homomorphism $f$ we get that $\Pi(\iota(s), \iota(t); m_{st}) = \Pi(\iota(t), \iota(s); m_{st})$ in the dihedral Artin group $A_{\iota(s) \iota(t)}$. By Lemma~\ref{lem:alternating product equality}, this forces $m_{\iota(s) \iota(t)}$ to be a divisor of $m_{st}$. This finishes the proof.
\end{proof}

\begin{proof}[Proof of Corollary~\ref{cor:injective_labels}]
    By Theorem \ref{thm:structure_hom}, we know that for every $s, t \in V(\Gamma_1)$, $m_{st}$ is a multiple of $m_{\iota(s) \iota(t)}$. Suppose that $m_{st} > m_{\iota(s) \iota(t)}$. There is a relation
    \[\underbrace{\iota(s) \iota(t) \iota(s) \cdots}_{m_{\iota(s) \iota(t)} \text{ times}} = \underbrace{\iota(t) \iota(s) \iota(t) \cdots}_{m_{\iota(s) \iota(t)} \text{ times}}.\]
    By injectivity, this forces in $A_{\Gamma}$ the equation
    \[\underbrace{sts \cdots}_{m_{\iota(s) \iota(t)} \text{ times}} = \underbrace{tst \cdots}_{m_{\iota(s) \iota(t)} \text{ times}}.\]
    This gives a non-trivial word of length $2 m_{\iota(s) \iota(t)} < 2 m_{st}$ representing the trivial element in the dihedral Artin subgroup $\langle s, t \rangle \leq A_{\Gamma}$. This contradicts \cite[Lemma 6]{appel1983artin}.
\end{proof}

\begin{proof}[Proof of Corollary~\ref{thm:hopf_cohopf}]
    Let $f: A_{\Gamma} \rightarrow A_{\Gamma}$ be a surjective homomorphism.
    Since $A_{\Gamma}$ is acylindrically hyperbolic as $|V(\Gamma)|\geq 3$ (see \cite{vaskou2021acylindrical}), $A_\Gamma$ does not virtually split as a direct product.
    Let $\Gamma_1, \Gamma_1'$ be as in Theorem~\ref{thm:structure_hom}.
    By surjectivity, we must have $S_1' = S$, hence $|S|\geq |S_1| = |S_1'| = |S|$. Thus, we have $|S|=|S_1|$. This means in particular that $S_1 = S$, and the same argument as in a previous proof shows that $f$ is, up to a conjugation and graph automorphism, the identity or the global inversion. In particular, it is an automorphism. 
    
    Let $f: A_{\Gamma} \rightarrow A_{\Gamma}$ be an injective homomorphism.
    Again, since $A_{\Gamma}$ is acylindrically hyperbolic as $|V(\Gamma)|\geq 3$, the image of $f$ does not virtually split as a direct product.
    Let $\Gamma_1, \Gamma_1'$ be as in Theorem~\ref{thm:structure_hom}.  By injectivity of $f$, we must have $S=S_0$ (where $S_0$ was defined in Lemma~\ref{lem:restriction_injective_S}). Moreover, if for some $s\neq s'$ we have $\Min(f(s)) = \Min(f(s'))$, then some powers of $f(s)$ and $f(s')$ would coincide, contradicting the injectivity of $f$. Thus, we have  $S_1 = S$, hence $|S|= |S_1| = |S_1'| \leq  |S|.$ Thus, we have $|S|=|S_1'|$. This means in particular that $S_1' = S$, and the same argument as in a previous, we get that $f$ is, up to a conjugation and graph automorphism, the identity or the global inversion. In particular, it is an automorphism. 
\end{proof}

Note that this result allows to obtain simple proofs for certain results previously proved in greater generality by Vaskou:

\begin{corollary}[\cite{vaskou2023automorphisms} Theorem A]
    Let $A_{\Gamma}$ be an XL-type free-of-infinity Artin group of rank at least $3$. Then the automorphism group $\Aut(A_{\Gamma})$ is generated by conjugations, graphs automorphisms, and the global inversion.
\end{corollary}

\begin{corollary}[\cite{vaskou2023isomorphism} Theorem B] \label{cor:rigidity}
    Let $A_{\Gamma_1}, A_{\Gamma_2}$ be two isomorphic XL-type free-of-infinity Artin groups. Then $\Gamma_1$ and $\Gamma_2$ are isomorphic as labelled graphs.
\end{corollary}

\section{The Cycle of Standard Trees Property in the XXXL case}\label{sec:cycle_trees}

The goal of this section is to prove the following: 

\begin{proposition} \label{prop:cycle_trees}
    $XXXL$-type Artin groups satisfy the Cycle of Standard Trees Property.
\end{proposition}

We will need the following weaker notion:

\begin{definition}
     A \textbf{loop of standard trees} is a sequence $T_1, \ldots, T_n$ of standard trees of the Deligne complex $D_\Gamma$ such that for every $i \in \mathbb{Z}/n\mathbb{Z}$, the intersection $T_i \cap T_{i+1}$ is a vertex, and such that: 
    \begin{itemize}
        \item for every $1\leq i<n$, the generators $x_i$ and $x_{i+1}$ of $\Fix(T_i)$ and $\Fix(T_{i+1})$ respetively generate a dihedral Artin group, 
        \item for every $1<|j-i|<n-1$, the generators of $\Fix(T_i)$ and $\Fix(T_j)$ generate a non-abelian free group.
    \end{itemize}
\end{definition}

Note that a cycle of standard trees is a loop of standard trees. However, we impose no condition on the subgroup generated by the generators of $\Fix(T_1)$ and $\Fix(T_n)$. We will need this more general notion of loop of standard trees in our proof of Proposition~\ref{prop:cycle_trees}, as certain subdiagrams we will be considering will naturally be bounded by loops of standard trees that are not necessarily cycles of standard trees. 

We will first consider the case where the associated loop $\gamma$ is embedded in $D_\Gamma$: 

\begin{lemma}\label{lem:cycle_trees_embedded_case}
    Let $A_\Gamma$ be an XXXL-type Artin group, and let $T_1, \ldots, T_n$ be a loop of standard trees in $D_\Gamma$. Assume that the corresponding loop $\gamma$ bounds a minimal disc diagram $D \rightarrow D_\Gamma$ that is non-singular (i.e. $D$ is homeomorphic to a disc). Then there exists a (necessarily unique) element $g\in A_\Gamma$ such that $\gamma$ is contained in the translate $gK_{\Gamma}$  of the fundamental domain of $D_\Gamma$. 
\end{lemma}

\begin{remark}
    Note that Lemma \ref{lem:cycle_trees_embedded_case} implies that the mentioned loop of standard trees is actually a cycle of standard trees.
\end{remark}

Let us explain how Proposition~\ref{prop:cycle_trees} follows from this. We have the following:

\begin{lemma}\label{lem:cycle-trees-embedded}
    Let $D \rightarrow D_\Gamma$ be a minimal disc diagram bounded by $\gamma$. Then $D$ is non-singular.
\end{lemma}

\begin{proof}
    If that were not the case, then in particular $\gamma$ would not be embedded, and we can consider a non-singular diagram $D'$ of $D$ that contains exactly one cut-vertex of $D$. The restriction $D' \rightarrow D_\Gamma$ defines an embedded sub-loop $\gamma'$ of $\gamma$, and has a corresponding loop of standard trees $T_i, \ldots, T_j$ with $j>i$. (Note that necessarily $j>i+1$ since standard trees are convex in $D_\Gamma$.)  Note that $T_i, \ldots, T_j$ then satisfies the assumptions of Lemma~\ref{lem:cycle_trees_embedded_case}, and it follows that $\gamma'$ is contained in a fundamental domain of $D_\Gamma$. But this now implies that the generator of $\Fix(T_i)$ and $\Fix(T_j)$ generate a dihedral Artin group, contradicting the assumption on the original cycle of standard trees that $\Fix(T_i)$ and $\Fix(T_j)$ generate a free subgroup since $|j-i|>1$. 
\end{proof}

\begin{proof}[Proof of Proposition~\ref{prop:cycle_trees}]
    This proposition is now a direct consequence of Lemmas~\ref{lem:cycle_trees_embedded_case} and~\ref{lem:cycle-trees-embedded},
\end{proof}

It now remains to prove Lemma \ref{lem:cycle_trees_embedded_case}.
Assume by contradiction that Lemma \ref{lem:cycle_trees_embedded_case} does not hold for some $\gamma$, and pick such a $\gamma$ such that the area of $D$ (i.e. its number of polygons) is minimal. We will prove that such a configuration is incompatible with the combinatorial Gauss--Bonnet Theorem \cite{ballmannbuyalo1996}.

By hypothesis, the disc diagram $D$ is non-singular.
To apply Gauss--Bonnet, we must first choose a set of angles for each triangle that appears in $D$.
Note that the map $D \rightarrow D_{\Gamma'}$ is a homeomorphism on each triangle, so we can talk of vertices of type $0, 1, 2$ in $D$.
For the rest of this section, we will make the following choices: 
\begin{itemize}
    \item the angle at a type $0$ vertex of a triangle $\Delta$ of $D$ is equal to $\angle_v(\Delta)\coloneqq \pi/3$,
    \item the angle at a type $1$ vertex of a triangle $\Delta$ of $D$ is equal to $\angle_v(\Delta)\coloneqq\pi/2$,
    \item the angle at a type $2$ vertex of a triangle $\Delta$ of $D$ is equal to $\angle_v(\Delta)\coloneqq\pi/6$.
\end{itemize}

With this set, triangles all have zero curvature, and curvature is localised only on vertices.
It will be easier to think of $D$ as having a polygonal structure.
The link of each type $0$ vertex of $D$ is a bipartite cycle made of type $1$ and $2$ vertices.
So, for every type $0$ vertex $v$, we can imagine we ``erase'' it and replace its star with a single polygon $P_v$ on that link.
We can do this because $\gamma$ is contained in a union of standard trees and does not pass through a vertex of type $0$.

\begin{figure}[H]
\centering
\includegraphics[scale=0.9]{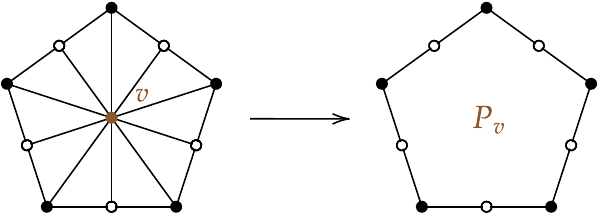}
\caption{Going from a triangular structure to a polygonal structure by erasing type $0$ vertices. Type $2$ vertices are drawn in black and type $1$ vertices in white.}
\label{FigurePolygonStructure}
\end{figure}

With the above choice of angles, we can define the combinatorial curvature $\kappa(-)$ of a vertex and of a triangle as in \cite{ballmannbuyalo1996}.
Note that by construction, triangles have zero curvature, and for vertices we have
\[\kappa(v) = 2\pi-\pi\cdot \chi(lk(v))-\sum\limits_{\substack{\Delta \text{ triangle} \\ \text{containing } v}} \angle_v(\Delta).\]

We can also define curvatures for this new polygonal structure. The curvature of a type $1$ or $2$ vertex remains unchanged, and we set the curvature of the polygon $P_v$ to be equal to $\kappa(v)$. 

We will assume by contradiction that $\gamma$ is not contained in a fundamental domain, in particular $D$ is not a single polygon. Let us denote by $V(D)$ the set of vertices of $D$, and by $P(D)$ the set of polygons of $D$.
With this polygonal structure, the combinatorial Gauss--Bonnet theorem \cite{ballmannbuyalo1996} gives us
\[\sum\limits_{v \in V(D)} \kappa(v) + \sum\limits_{P \in P(D)} \kappa(P) = 2\pi. \]

Because our Artin groups are of large type, internal vertices of type $1$ and $2$-cells have non-positive curvature by construction, and internal vertices of type $2$ always have non-positive curvature by Appel--Schupp \cite[Lemma 6]{appel1983artin}.
Moreover, because the boundary of $D$ is mapped to (convex) standard trees, the only cells of $D$ that can contribute positive curvature are vertices of type $2$ in $\partial D$ that correspond to intersection points between consecutive standard trees. 

\medskip

\noindent \textbf{Notation:} For a vertex $v$ of $D$, we denote by $n_v$ the number of polygons of $D$ containing~$v$. 

\begin{lemma} \label{lem:at_least_5_cells}
    Let $v$ be the vertex  $T_i \cap T_{i+1}$ for some $1 \leq i <n$. We either have $n_v =1$, or $n_v \geq 5$. In the latter case, we have $\kappa(v) \leq 0$.
\end{lemma}

\begin{proof}
Let $x_i, x_{i+1}$ be generators of $\Stab(T_i)$, $\Stab(T_{i+1})$ (up to taking inverses, we assume that $x_i$ and $x_{i+1}$ are conjugated to standard generators). Since they generate a dihedral Artin group by assumption, Lemma \ref{lem:subgroup_conjugates_generators} implies that the pair $\{x_i, x_{i+1}\}$ is conjugated to a pair of standard generators.
In particular, there exist edges of $T_i$ and $T_{i+1}$ that make an angle of precisely $\pi/3$.
Recall also that two distinct adjacent edges of a standard tree make an angle at least $6\pi/3$ since our Artin groups are XXXL-type, by Martin--Przytycki \cite[Corollary 4.2]{martin2022acylindrical}.
So, the two edges of $D$ containing $v$ and contained in $T_{i}$ and $T_{i+1}$ either make an angle of $\pi/3$ or at least
$5\pi/3$
by the triangle inequality. The result follows.
\end{proof}

We call a vertex $v = T_i \cap T_{i+1}$ of $\partial D$ a \textbf{corner} if $n_v =1$, and an \textbf{almost-corner} if $n_v \geq 5$.
Let $v_0\coloneqq T_1\cap T_n$, which we will call the \textbf{basepoint} of $D$. Note that the previous lemma does not apply to $v_0$, in particular we could have $n_{v_0}=1$ (corner), $n_{v_0}\geq 5$ (almost-corner) but also $n_{v_0} = 2$, $3$, $4$. 

A \textbf{corner-cell} will be a polygon of $D$ that contains at least one corner of $D$.
Note that it may contain several.
Also note that, by minimality of $D$ and since $D$ is not a single polygon, every corner-cell $P$ is such that $P \cap \partial D$ is connected and homeomorphic to an interval.
Since corners and possibly the basepoint are the only cells of $D$ contributing positive curvature (with the basepoint contributing at most $\pi/3$ if it is not a corner), the combinatorial Gauss--Bonnet Theorem implies that corners and hence corner-cells exist.

Let $P$ be a corner-cell, and let $u, u'$ be the two extremal vertices of the path $P\cap \partial D$, which will be called \textbf{end-vertices} of $P$. 
These two vertices subdivide $\partial P$ into two paths, the outer path $\partial_o P$ contained in $\partial D$, and the inner path $\partial_i P$ such that $\partial_i P \cap \partial D = \{u, u'\}$.

\begin{figure}[H]
\centering
\includegraphics[scale=1]{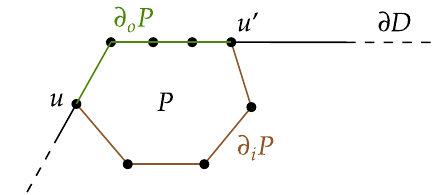}
\caption{The inner and outer paths of a corner-cell $P$.}
\label{FigureOuterInnerPaths}
\end{figure}

A priori, $u, u'$ can be either type $1$ or type $2$ vertices. We first need a small result:

\begin{lemma} \label{LemmaInnerPath}
    The inner path of a corner-cell contains at least two type $2$ vertices. 
\end{lemma}

\begin{proof}
     Since $D$ is nonsingular by assumption, $\partial_i P$ contains at least one edge, hence at least one type $2$ vertex. If we had only one type $2$ vertex, then this would mean that two non-consecutive standard trees $T_i$ and $T_j$ intersect in a fundamental domain (see Figure \ref{Fig:Type2inInnerPath}). In particular, the generators of their fixed-point sets would generate a dihedral Artin group, contradicting our assumption on $\gamma$ that they generate a free subgroup.
    
    \begin{figure}[H]
    \centering
    \includegraphics[scale=1]{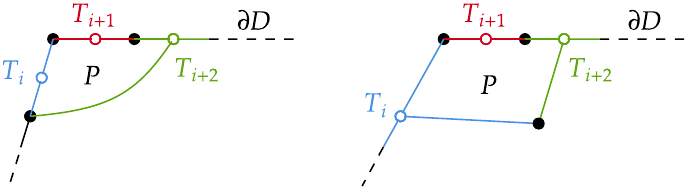}
    \caption{Two configurations illustrating the case where $\partial_i P$ contains only one type $2$ vertex. Type $2$ vertices are drawn in black and type $1$ vertices in white.}
    \label{Fig:Type2inInnerPath}
    \end{figure}
\end{proof}

For each such corner-cell $P$, we call the first and last type $2$ vertices on the embedded segment $\partial_i P$ the \textbf{special} vertices of $P$.

\begin{lemma} \label{LemmaSpecialVertices}
    Let $v$ be a type $2$ vertex of $D$. Then $v$ is a special vertex of at most two corner-cells.
\end{lemma}

\begin{proof}
    If $v \in \partial D$, then since $D$ is non-singular and the intersection of a corner-cell with $\partial D$ is homeomorphic to an interval, $v$ can be a special vertex of at most two corner-cells.
    
    Now, assume that $v$ is contained in at least $3$ such corner-cells and that $v$ is internal.
    By assumption, $v$ is at distance $1$ from the boundary part of each of the corner cells, so for every corner cell $P$ having $v$ as a special vertex there is an end-vertex at distance $1$ from $v$.
    Note that these vertices belong to $\partial D$, so each of them is contained in at most two of the corner-cells.
    So there must exist a path of two edges $u-v-u'$ in $D$ where $u, u' \in \partial D$ are distinct vertices of type $1$ that are end-vertices of corner-cells (see Figure \ref{Fig:subdiagramD'}).
    Let $T_i$, $T_j$ be the standard trees of the cycle of standard trees such that $u \in T_i$ and $u' \in T_j$ (observe that $v$ belongs to both standard trees).
    Without loss of generality we can assume $j >i$. Note $T_i, \ldots, T_j$ form another loop of standard trees.
  
    Let $D'$ be the corresponding strict subdiagram of $D$, and $\gamma'$ the corresponding loop in $D_\Gamma$.
    Since $D$ is non-singular and $v$ is internal, it follows that $D'$ is also non-singular. 

    \begin{figure}[H]
    \centering
    \includegraphics[scale=1]{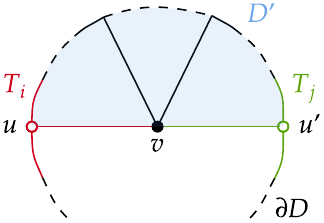}
    \caption{Subdiagram $D'$, delimited by $u-v-u'$ and part of $\partial D$.}
    \label{Fig:subdiagramD'}
    \end{figure}

    Thus, $D'$ and $\gamma'$ satisfy the hypotheses of Lemma~\ref{lem:cycle_trees_embedded_case}.
    By minimality of $\gamma$, we thus have that this sub-cycle bounds a single polygon contained in a fundamental domain of $D_\Gamma$. But that now implies that the generators of $\Fix(T_i)$ and $\Fix(T_j)$ generate a dihedral Artin group, contradicting the hypothesis on $\gamma$ that they generate a free group. 
\end{proof}

We will now redistribute curvature:
\begin{itemize}
    \item If $v$ is a special vertex of some corner-cell, set 
    \[\kappa'(v)\coloneqq  \kappa(v) + n\pi/3,\]
    where $n\leq 2$ is the number of corner-cells containing $v$ as a special vertex.
    \item If $P$ is a corner-cell of $D$, set 
    \[\kappa'(P) \coloneqq \kappa(P) - 2\pi/3\]
    \item For all other vertices and polygons, set $\kappa'(v) \coloneqq \kappa(v)$ and $\kappa'(P) \coloneqq \kappa(P)$ respectively.
\end{itemize}

\begin{lemma}\label{lem:redistribution_curvature}
    We have 
    \[\sum\limits_{v \in V(D)} \kappa'(v) + \sum\limits_{P \in P(D)} \kappa'(P) = 2\pi\]
\end{lemma}

\begin{proof}
    This is a direct consequence of the way we have redistributed curvature. For each pair $(P, v)$ where $P$ is a corner-cell and $v$ is a special vertex of $P$, $v$ gains $\pi/3$, while $P$ loses $\pi/3$ curvature. Thus we have 
    \[\sum\limits_{v \in V(D)} \kappa'(v) + \sum\limits_{P \in P(D)} \kappa'(P) = \sum\limits_{v \in V(D)} \kappa(v) + \sum\limits_{P \in P(D)} \kappa(P)\]
    which is equal to $2\pi$ by the combinatorial Gauss--Bonnet Theorem \cite{ballmannbuyalo1996}.
\end{proof}

\begin{lemma}
Let $v$ be a vertex of $P$. Then: 
\begin{itemize}
    \item  If $v$ is a corner of $D$, we have $\kappa'(v) = 2\pi/3$. 
    \item If $v$ is the basepoint of $D$, and $v$ is not a corner, then we have $\kappa'(u) \leq \pi$.
    \item For every other vertex $v$ we have $\kappa'(v) \leq 0$.
\end{itemize}
\end{lemma}

\begin{proof}
    Recall that with our choice of angles, we had $\kappa(v) \leq 0$ for every vertex $v \in D$ except when $v$ is a corner (where $\kappa(v) = 2\pi/3$) and possibly when $v$ is a basepoint (if it is not a corner, then $\kappa(v) \leq \pi/3$). Since a corner cannot be a special vertex, we have $\kappa'(v) = \kappa(v) = 2\pi/3$ for a corner $v$. If $v$ is the basepoint and is not a corner, then by Lemma \ref{LemmaSpecialVertices} \[\kappa'(v) \leq \kappa(v) + 2\pi/3 \leq \pi.\]
    It remains to consider the case of a vertex that is neither a corner nor the basepoint. If $v$ is not special, then $\kappa'(v) = \kappa(v) \leq 0$. Otherwise, notice that since our Artin group $A_\Gamma$ is of XXXL-type by assumption, $v$ is contained in either at least $5$ polygons if $v \in \partial D$ by Lemma~\ref{lem:at_least_5_cells}, or in at least $12$ polygons if $v$ is internal by \cite[Lemma~6]{appel1983artin}. In each case, the choice of angles gives us $\kappa(v) \leq -2\pi/3$, hence 
    \[\kappa'(v) \leq \kappa(v) + 2\pi/3 \leq 0\]
    which concludes the proof
\end{proof}

\begin{lemma}
    For every polygon $P$ of $D$ we have $\kappa'(P) \leq 0$.
\end{lemma}

\begin{proof}
    This is a direct consequence of the construction of $\kappa'$, and the fact that $\kappa(P)\leq 0$ for every polygon of $D$, by our choice of angles.
\end{proof}

\begin{lemma}
    Let $P$ be a corner-cell, and let $v_1, \ldots, v_k$ be the corners contained in $P$. Then 
    \[\kappa'(P) + \sum_{1 \leq i \leq k} \kappa'(v_i)\leq 0\]
\end{lemma}

\begin{proof}
By Lemma \ref{LemmaInnerPath}, there are at least two type $2$ vertices of $P$ that lie in $\partial_i P$. Let us denote by $n_i(P)\geq 2$ the number of type 2 vertices in $\partial_i P$, and by $n_c(P) =k$ the number of corners of $D$ contained in $P$. We thus have, by our choice of angles:

\[\kappa'(P) = \kappa(P) - \frac{2\pi}{3} \leq 2\pi - (n_c(P) + n_i(P)) \frac{2\pi}{3} - \frac{2\pi}{3}, \]

\noindent and hence 
\[\kappa'(P) + \sum_{1 \leq i \leq k} \kappa'(v_i) = 2\pi - n_i(P)\frac{2\pi}{3} - \frac{2\pi}{3} \leq 0.\]
\end{proof}

\begin{proof}[Proof of Lemma~\ref{lem:cycle_trees_embedded_case}] 
Recall that by Lemma \ref{lem:redistribution_curvature} we have that 
\[\sum\limits_{v \in V(D)} \kappa'(v) + \sum\limits_{P \in P(D)} \kappa'(P) = 2 \pi.\]
Since a corner belongs to a single corner-cell by definition, the above and the previous lemmas implies that

\[2 \pi \leq \kappa'(u_0) + \sum\limits_{\substack{P \in P(D) \\ \text{corner-cell}}} \underbrace{\big(\kappa'(P) + \sum\limits_{\substack{v \in V(D) \\ \text{corner of } P}} \kappa'(v) \big)}_{\leq 0} + \sum\limits_{\substack{X \in V(D) \sqcup P(D) \\ \text{not corner,} \\ \text{not corner-cell}}} \underbrace{\kappa'(X) }_{\leq 0} \leq \pi,\]
which gives a contradiction.
\end{proof}

\section{Graphs with no cut-vertex or separating edge}\label{sec:no_separating}

The goal of this section is to prove the following intermediate result on our way to Theorem~\ref{thm:XXXL_cohopf}: 

\begin{proposition}\label{prop:inj_hom_general_case}
    Let $f: A_\Gamma \hookrightarrow A_{\Gamma'}$ be an injective homomorphism between XL-type Artin groups.
    Suppose that $\Gamma$ is connected with $|V(\Gamma)|\geq 3$ and has no cut-vertex or separating edge. Suppose also that $A_{\Gamma'}$ satisfies the Cycle of Standard Trees Property.

    Then there exists an element $g\in A_{\Gamma'}$, an embedding $\varphi: \Gamma \hookrightarrow \Gamma'$ of labelled graphs, and an integer $\varepsilon = \pm 1$ such that $f(v) = g\varphi(v)^\varepsilon g^{-1}$ for all standard generators $v \in V(\Gamma)$. 
\end{proposition}

We first treat the case where $\Gamma$ is a cycle.

\begin{lemma}\label{lem:inj_hom_cycle_case}
    Let $f: A_\Gamma \hookrightarrow A_{\Gamma'}$ be an injective homomorphism between XL-type Artin groups of rank $\geq 3$. Suppose that $\Gamma$ is a cycle. Suppose also that $A_{\Gamma'}$ satisfies the Cycle of Standard Trees Property.
    
    Then there exists an element $g\in A_{\Gamma'}$, an embedding $\varphi: \Gamma \hookrightarrow \Gamma'$ of labelled graphs, and an integer $\varepsilon = \pm 1$ such that $f(v) = g\varphi(v)^\varepsilon g^{-1}$ for all standard generators $v \in V(\Gamma)$.  
\end{lemma}

\begin{proof}
By Lemma~\ref{lem:generators_to_generators}, we know that every standard generator of $A_{\Gamma}$ is sent to a conjugate of a standard generator of $A_{\Gamma'}$ or its inverse. 

Denote cyclically by $\{v_i \}_{i \in \mathbb{Z} / n \mathbb{Z}}$ the vertices of $\Gamma$.
For each $i$, let $T_i$ be the standard tree of $D_{\Gamma'}$ that is the fixed-point set of $f(v_i)$. Since $f$ is injective, for every $i\in \mathbb{Z} / n \mathbb{Z}$, the elements $f(v_i)$ and $f(v_{i+1})$ generate a dihedral Artin group, hence they fix a common vertex of type $2$ in $D_{\Gamma'}$ by Lemma~\ref{lem:abstract_dihedral_subgroups}.
In particular, the trees $T_i$ and $T_{i+1}$ intersect.
Thus, $T_1, \ldots, T_n$ forms a cycle of standard trees.
Moreover, by Theorem~\ref{thm:vdL}, for every $i, j$ with $1<|j-i|<n-1$, the standard generators $v_i$ and $v_j$ generate a free subgroup, hence so do $f(v_i)$ and $f(v_j)$ by injectivity of $f$.
We can thus apply Lemma \ref{lem:trivial_labels}, and it follows that there exists an element $g \in A_{\Gamma'}$, and a map $\varphi: V(\Gamma) \rightarrow V(\Gamma')$ such that for every vertex $s$ of $\Gamma$, we have that \[f(s) = g\varphi(v)^{n_v}g^{-1}.\]
Notice that $\varphi$ is injective. Indeed, if that were not the case, there would exist two distinct vertices $s, t$ of $\Gamma$ such that $f(s)$ and $f(t)$ have a common a power. But since $\langle s \rangle \cap \langle t \rangle = \{1\}$ by Theorem~\ref{thm:vdL}, this contradicts the injectivity of $f$. 

Now, let $s, t \in V(\Gamma)$ be two adjacent vertices. Up to conjugation by $g$, the relation $\Pi(s,t;m_{st}) = \Pi(t,s;m_{st})$ through $f$ becomes
\[\Pi(\varphi(s)^{n_s},\varphi(t)^{n_t};m_{st}) = \Pi(\varphi(t)^{n_t},\varphi(s)^{n_s};m_{st}).\]
By Lemma \ref{lem:alternating product equality}, it must then be that $n_s = n_t = \pm 1$. This shows that whenever $s$ and $t$ are adjacent, $n_s$ and $n_t$ are equal and both are either $1$ or $-1$. As $\Gamma$ is connected, then all $n_{v_i}$'s are equal to some $\varepsilon = \pm 1$.

Let us now prove that $\varphi$ extends to an embedding of labelled graph $\varphi: \Gamma \hookrightarrow \Gamma'$. We only have left to show that $\varphi$ maps adjacent vertices of $\Gamma$ to adjacent vertices of $\Gamma'$. Let $s, t \in V(\Gamma)$ be two adjacent vertices. By injectivity of $f$, we have that the $f(s)$ and $f(t)$ generate a dihedral Artin group. In particular, the subgroup generated by $\varphi(s)$ and $\varphi(t)$ cannot be free, and it follows from Theorem~\ref{thm:vdL} that $\varphi(s)$ and $\varphi(t)$ are adjacent in $\Gamma'$. Finally, let us show that $\varphi$ preserves the labels of edges. Let $s, t \in V(\Gamma)$ be two adjacent vertices. Up to conjugation by $g$, the relation $\Pi(s,t;m_{st}) = \Pi(t,s;m_{st})$ through $f$ becomes
\[\Pi(\varphi(s),\varphi(t);m_{st}) = \Pi(\varphi(t),\varphi(s);m_{st}).\]
By Lemma \ref{lem:alternating product equality}, we have that $m_{st}$ is a multiple of $m_{\varphi(s)\varphi(t)}$. Up to conjugation by $g$ and composition with the global inversion, we can assume that $f$ sends $s$ to $\varphi(s)$ and $t$ to $\varphi(t)$. If $m_{st}>m_{\varphi(s)\varphi(t)}$, then the element \[\Pi(s,t;m_{\varphi(s)\varphi(t)})\Pi(t,s;m_{\varphi(s)\varphi(t)})^{-1} \]
 is a non-trivial element of $\langle s, t \rangle < A_\Gamma$ that maps to the identity element of $A_{\Gamma'}$, contradicting the injectivity of $f$. 
 Thus, $\varphi$ extends to a label-preserving embedding $\Gamma \rightarrow \Gamma'$, which concludes the proof. 
\end{proof}

\begin{proof}[Proof of Proposition~\ref{prop:inj_hom_general_case}]
    Let $\gamma$ be an induced cycle of $\Gamma$ (which exists since $\Gamma$ does not have a cut-vertex).
    By applying Lemma~\ref{lem:inj_hom_cycle_case}, we get an element $g_\gamma\in A_{\Gamma'}$ and a label-preserving graph embedding $\varphi_\gamma: \gamma \hookrightarrow \Gamma'$ such that for every standard generator $v\in V(\gamma)$, we have \[f(v)=g_\gamma\varphi_\gamma(v)g_\gamma^{-1}\]
    By Lemma~\ref{lem:trivial_labels}, we have $g_\gamma = g_{\gamma'}$ for every two induced cycles $\gamma, \gamma'$.
    Moreover, the embeddings $\varphi_\gamma$ are compatible on their intersection and yield a label-preserving graph homomorphism $\varphi:\Gamma \rightarrow \Gamma'$.
    By injectivity of $f$, it follows that $\varphi$ is an embedding, which concludes the proof. 
\end{proof}

\paragraph{Consequences.} We list below a few immediate corollaries:

\begin{corollary}\label{cor:cst_cohopfian}
    Let $A_\Gamma$ be an XL-type Artin group that satisfies the Cycle of Standard Trees Property. Suppose that $\Gamma$ is connected with $|V(\Gamma)|\geq 3$ and has no cut-vertex or separating edge.
    Then $A_\Gamma$ is co-hopfian.
\end{corollary}

\begin{proof}
    Let $f: A_\Gamma \hookrightarrow A_\Gamma$ be an injective homomorphism. By Proposition~\ref{prop:inj_hom_general_case}, we get a label-preserving embedding $\varphi: \Gamma \hookrightarrow \Gamma$, an integer $\varepsilon = \pm1$ and an element $g \in A_{\Gamma'}$ such that for every standard generator $v \in V(\Gamma)$, we have 
    \[f(v) = g\varphi(v)^{\pm1}g^{-1}\]
    Since $\Gamma$ is a finite graph, it follows that $\varphi$ is an isomorphism. In particular, the image of $f$ contains the $g$-conjugates of all the standard generators of $A_{\Gamma}$. Thus, $f$ is surjective, hence an isomorphism. 
\end{proof}

\begin{remark}\label{rem:recovering_Vaskou}
    Note that we can use Proposition~\ref{prop:inj_hom_general_case} to recover a particular case of Vaskou's solution of the Isomorphism Problem for large-type Artin groups \cite{vaskou2023isomorphism}.
    Indeed, if we have an isomorphism $f: A_\Gamma \rightarrow A_{\Gamma'}$ between two XXXL-type Artin groups where $\Gamma$ is connected, is not an edge and does not containing a cut-vertex or separating edge, then Proposition~\ref{prop:inj_hom_general_case} yields a label-preserving embedding $\varphi: \Gamma \hookrightarrow \Gamma'$, an integer $\varepsilon = \pm1$ and an element $g \in A_{\Gamma'}$ such that for every standard generator $v \in V(\Gamma)$, we have 
    $$f(v) = g\varphi(v)^{\pm1}g^{-1}$$
    Since $f$ is surjective and the image of $f$ is contained in the parabolic subgroup $g \langle \varphi(v) ~|~ v \in V(\Gamma) \rangle g^{-1}$, it follows that $\varphi$ is surjective. Thus, it is a label-preserving isomorphism. 
    
    When $\Gamma'=\Gamma$, we obtain a generalisation of a result of Vaskou \cite{vaskou2023automorphisms}, namely that $\Aut(A_\Gamma)$ is generated by the conjugations, graph automorphisms, and the global inversion.
\end{remark}

\begin{corollary}\label{cor:characterisation_finite_Out}
   Let $A_\Gamma$ be an XL-type Artin group that satisfies the Cycle of Standard Trees Property. Then $\Out(A_\Gamma)$ is finite if and only if $\Gamma$ is connected, is not an edge, and does not contain a cut-vertex or a separating edge. Moreover, when $\Out(A_{\Gamma})$ is finite, $\Aut(A_{\Gamma})$ is generated by the conjugations, the graph automorphisms, and the global inversion.
\end{corollary}

\begin{proof}
    If $\Gamma$ is connected, not an edge, and has no cut-vertex or separating edge, then the result follows from Remark~\ref{rem:recovering_Vaskou}
    as in that case one checks that $\Out(A_{\Gamma}) \cong Aut(\Gamma) \times \mathbb{Z} / 2 \mathbb{Z}$, which is finite. 

    If $\Gamma$ is disconnected, then $A_\Gamma$ is a free product $A_{\Gamma_1} * A_{\Gamma_2}$.
    If both $\Gamma_1$ and $\Gamma_2$ are a single vertex, then $A_\Gamma$ is a free group of rank $2$, and so $\Out(A_\Gamma)$ is infinite.
    If $\Gamma_1$ is not a single vertex, then $A_{\Gamma_1}$ has an non-central element $g$ that is infinite order because $A_\Gamma$ is large-type.
    Thus, since $g^n$ cannot commute with $A_{\Gamma_2}$ for any $n\in\mathbb{Z}$, the automorphism being the identity on $A_{\Gamma_1}$ and the conjugation by $g$ on $A_{\Gamma_2}$ defines an infinite order element of $\Out(A_\Gamma)$.
    
    If there is a cut-vertex or separating edge, then the graph $\Gamma$ decomposes as a union $\Gamma = \Gamma_1 \cup \Gamma_2$ of two induced subgraphs such that $\Gamma_1\cap \Gamma_2$ is either a vertex or an edge.
    Let $z \in A_{\Gamma_1\cap \Gamma_2}$ be a non-trivial element that is central in $A_{\Gamma_1\cap \Gamma_2}$, and consider the partial conjugation $f: A_\Gamma \rightarrow A_\Gamma$ that is the identity on $A_{\Gamma_1}$ and the conjugation by $z$ on $A_{\Gamma_2}$.
    Then one checks by contradiction that for every $n$, the iterate $f^n$ is not a conjugation. Indeed, if some $f^n$ was the conjugation by an element $x\in A_{\Gamma}$, then since $f^n$ preserves each $A_{\Gamma_i}$, we would have that $x$ normalises both $\Gamma_i$.
    Since parabolic subgroups on more than one generators are self-normalising by \cite[Theorem E]{cumplido2022parabolic}, it follows that $x \in A_{\Gamma_1 \cap \Gamma_2}$.
    Since $f^n$ is the identity on $A_{\Gamma_1}$, it follows that $x$ is central in $A_{\Gamma_1}$.
    Moreover, $f^n$ is a non-trivial conjugation on $A_{\Gamma_2}$, we would have that $x$ is central in $A_{\Gamma_1}$ and non-trivial.
    In particular, since large-type Artin groups on more than two generators have trivial centre by \cite{Godelle2007}, it follows that $\Gamma_1$ is an edge and $\Gamma_1 \cap \Gamma_2$ is a vertex. But in the dihedral Artin group $A_{\Gamma_1}$, powers of standard generators cannot be central, a contradiction.
\end{proof}

\section{The co-Hopf property} \label{sec:coHopf}

The goal of this section is to prove the following:

\begin{theorem}\label{thm:cohopf_XXXL_characterisation}
Let $A_\Gamma$ be an XL-type Artin group that satisfies the Cycle of Standard Trees Property. Then $A_\Gamma$ is co-hopfian if and only if $\Gamma$ is connected with $|V(\Gamma)|\geq 3$, and does not contain a cut-vertex. 
\end{theorem}

We start with one of the implications, which holds in greater generality: 

\begin{lemma}
    Let $A_{\Gamma}$ be an Artin group. If $\Gamma$ contains a cut-vertex, then $A_\Gamma$ is not co-hopfian.
\end{lemma}

\begin{proof}
    Let $c$ be a cut-vertex of $\Gamma$, and let $\Gamma_1, \Gamma_2$ be induced subgraphs of $\Gamma$ such that $\Gamma_1 \cup \Gamma_2 = \Gamma$ and $\Gamma_1 \cap \Gamma_2 = \{c\}$.
    This induces a decomposition of $A_\Gamma$ as an amalgamated product 
    $$A_\Gamma = A_{\Gamma_1} *_{\langle c \rangle} A_{\Gamma_2}$$
    and we denote by $T$ the corresponding Bass--Serre tree.
    Let $e \subset T$ be the fundamental edge, with vertices $v_1, v_2$ with stabilisers $A_{\Gamma_1}, A_{\Gamma_2}$ respectively. 
    
    Let $a_1, a_2$ be neighbours of $c$ in $\Gamma_1, \Gamma_2$ respectively.
    For $i =1, 2$, we pick an element $h_i \in A_{\Gamma_i} \backslash \langle v \rangle$ that centralises $v$. (For instance, if $m_{va_i} \geq 3$, pick $h_i:= z_{va_i}$. If $m_{va_i} =2$, pick $h_i:= a_i$)

    For $i=1, 2$, let $e_i := h_ie$, let $w_1:= h_2v_1$ and $w_2:=h_1v_2$. Let $\gamma:= e_1 \cup e \cup e_2$ be the geodesic path with end vertices $w_1, w_2$. Note that the stabilisers of $w_1, w_2$ are $H_1 := A_{\Gamma_1}^{h_2}$ and $H_2 :=A_{\Gamma_2}^{h_1}$ respectively. We will prove that the subgroup $H:= \langle H_1, H_2 \rangle$ is a strict subgroup of $A_\Gamma$ isomorphic to $A_\Gamma$. The pointwise stabiliser of $\gamma$ is 
    $$\Stab(w_1) \cap \Stab(w_2) = \Stab(e_1) \cap \Stab(e) \cap \Stab(e_2)  = \langle c \rangle$$
    by construction of $h_1, h_2$. For each $i =1, 2$, it follows that for $g \in \Stab(w_i)$, we either have $g\gamma = \gamma$ (if $g \in \langle c \rangle$) of $g\gamma \cap \gamma = \{ w_i \}$ (otherwise).
    Let $T'$ be the minimal subtree of $T$ that contains all the $H$-translates of $\gamma$.
    Note that by the previous discussion, the vertices $v_1, v_2$ have valence $2$ in $T'$ whereas they have infinite valence in $T$.
    Since $A_\Gamma$ acts transitively on the edges of $T$, it follows that $H$ is a strict subgroup of $A_\Gamma$.
    Moreover, we can define a new simplicial structure on $T'$ by removing all vertices of $T'$ of valence $2$, i.e. the $H$-orbits of $v_1$ and $v_2$.
    For that new structure, it follows that $H$ acts transitively on the edges of $T'$, with strict fundamental domain the edge between $w_1$ and $w_2$.
    It now follows from Bass-Serre theory that $H$ is isomorphic to the amalgamated product \[H_1 *_{\langle c \rangle} H_2 = A_{\Gamma_1}^{h_2} *_{\langle c \rangle} A_{\Gamma_2}^{h_1}\]
    Note that we have an isomorphism $f_1: A_{\Gamma_1} \rightarrow  A_{\Gamma_1}^{h_2}$ induced by the conjugation by $h_2$ in $A_\Gamma$.
    Similarly, we have an isomorphism $f_2: A_{\Gamma_2} \rightarrow A_{\Gamma_2}^{h_1}$ induced by the conjugation by $h_1$ in $A_\Gamma$.
    Moreover, $f_1$ and $f_2$ agree on the amalgamating subgroup $\langle c \rangle$, as they both induce the identity on $\langle c \rangle$ by construction of $h_1, h_2$.
    Thus, these two isomorphisms can be amalgamated into an isomorphism  $f_1*f_2: A_{\Gamma_1} *_{\langle c \rangle} A_{\Gamma_2} \rightarrow A_{\Gamma_1}^{h_2} *_{\langle c \rangle} A_{\Gamma_2}^{h_1}$, i.e. an isomorphism between $A_\Gamma$ and $H$.
    Hence, $A_\Gamma$ is not co-hopfian. 
\end{proof}

\begin{remark}
    We leave it to the reader to check that a similar (and simpler) reasoning shows that a free product of non-trivial groups is never co-hopfian. In particular, an Artin group over a disconnected graph $\Gamma$ is never co-hopfian.
\end{remark}

The condition of $\Gamma$ not being a single edge in Theorem~\ref{thm:cohopf_XXXL_characterisation} is due to the following result, which is formulated in \cite[Section 2]{bell2006braid} as a remark:

\begin{lemma}
    Dihedral Artin groups are not co-hopfian.
\end{lemma}

We now consider the second implication:

\begin{proposition}\label{prop:cohopf_hard_direction}
 Let $A_\Gamma$ be an XL-type Artin group that satisfies the Cycle of Standard Trees Property. Suppose that $\Gamma$ is connected, is not an edge, and $\Gamma$ does not contain a cut-vertex. Then $A_\Gamma$ is co-hopfian. 
\end{proposition}

The proof of Proposition~\ref{prop:cohopf_hard_direction} will take the remainder of this section. We first explain how to decompose a general graph into chunks that have no separating edge or cut-vertex. 

\begin{definition}[\cite{an2022automorphism}]
Let $\Gamma$ be a finite simplicial graph without a cut-vertex. An induced subgraph $\Gamma'$ of $\Gamma$ is called a \textbf{chunk} if it is connected and maximal among all the induced subgraphs $\Gamma''$ of $\Gamma$ such that $\Gamma''$ does not have a cut-vertex or separating edge. We define the following graph $T$: 
\begin{itemize}
    \item vertices of $T$ are the chunks of $\Gamma$ and the separating edges of $\Gamma$.
    \item Edges correspond to inclusion of subgraphs. That is, for every chunk $\Gamma'$ of $\Gamma$ and every separating edge of $\Gamma$ contained in $\Gamma'$, we add an edge between $\Gamma'$ and $e$.
\end{itemize}
The graph $T_\Gamma$ is a simplicial tree, called the \textbf{chunk tree} of $\Gamma$. 
\end{definition}

Note that the graph $T_\Gamma$ is bipartite and encodes how $\Gamma$ can be written as an iterated amalgamation of its various chunks along separating edges.
The following lemma is immediate:

\begin{lemma}\label{lem:simple_chunk}
Let $\Gamma$ be a finite simplicial graph without cut-vertex and $|V(\Gamma)| \geq 3$, and let $\Gamma'$ be a chunk of $\Gamma$. Then $|V(\Gamma')| \geq 3$, and $\Gamma'$ does not contain cut-vertices or separating edges. 
\end{lemma}

We further have: 

\begin{lemma}
    Let $\Gamma$ be a finite connected simplicial graph with $|V(\Gamma)|\geq 3$ and without a cut-vertex. Then $\Gamma$ is the union of its chunks. \qedhere
\end{lemma}

\begin{proof}
    It is enough to prove that every edge of $\Gamma$ is contained in a chunk. Since $\Gamma$ is connected with $|V(\Gamma)|\geq 3$ and without a cut-vertex, it follows that every edge is contained in a cycle. Since a cycle is a connected subgraph on at least $3$ vertices without cut-vertex and separating edge, every edge is contained in at least one chunk by construction. 
\end{proof}

\begin{example} 
We draw below an example of a graph $\Gamma$, its chunks, and the corresponding chunk tree $T_\Gamma$: 

\begin{figure}[H]
\centering
\includegraphics[scale=1]{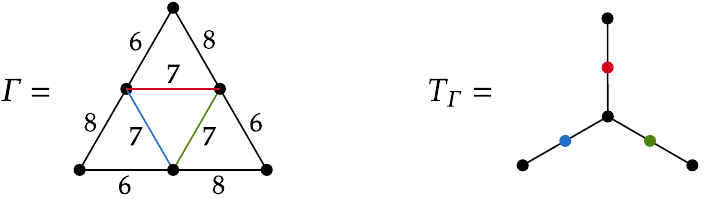}
\caption{On the left, a graph $\Gamma$ with $4$ chunks (each of them a triangle). On the right, the corresponding chunk tree $T_{\Gamma}$. The separating edges are the ones with label $7$.}
\label{FigureTriforce}
\end{figure}
\end{example}

\begin{convention}
    We will choose an ordering $\Gamma_0, \ldots, \Gamma_k$ of the chunks of $\Gamma$ such that for each $i\geq 1$, $\Gamma_i$ shares an edge (necessarily separating in $\Gamma$, by construction) with $\Gamma_0 \cup \ldots \cup \Gamma_{i-1}$. 
\end{convention}

   In the rest of this section $A_\Gamma$ will be an Artin group that satisfies the hypotheses of Proposition~\ref{prop:cohopf_hard_direction}.

\begin{lemma}\label{lem:induced_isom_chunks}
    For each chunk $\Gamma_i$ of $\Gamma$, there exists an element $g_i \in A_\Gamma$ and an embedding $\varphi_i: \Gamma_i \hookrightarrow \Gamma$ such that for every $v \in V(\Gamma_i)$, we have $f(v) = g_i\varphi_i(v)^{\pm 1}g_i^{-1}$ 
\end{lemma}

\begin{proof}
    Since $\Gamma_i$ has no cut-vertex or separating edge and is not an edge by Lemma~\ref{lem:simple_chunk}, this is a direct application of Proposition~\ref{prop:inj_hom_general_case} to the restriction $f_i: A_{\Gamma_i} \rightarrow A_\Gamma$.
\end{proof}

\begin{convention}
Up to conjugating $f$ by the element $g_0$, we will assume that $g_0 = 1$. 
\end{convention}

\begin{definition}\label{def:labelling_decomposition_tree}
Let $f: A_\Gamma \rightarrow A_{\Gamma'}$ be an isomorphism of two XL-type Artin groups, such that $\Gamma$ is connected, $|V(\Gamma)|\geq 3$, and $\Gamma$ does not have a cut-vertex. Further assume that $A_{\Gamma'}$ satisfies the Cycle of Standard Trees Property. For every chunk $\Gamma_i$ of $\Gamma$, we denote by $\alpha_f(\Gamma_i)$ the element of $A_{\Gamma'}$ obtained by applying Lemma~\ref{lem:induced_isom_chunks} to the restriction $f: A_{\Gamma_i} \rightarrow A_{\Gamma'}$.
\end{definition}

\begin{lemma}\label{lem:domains_sharing_two_trees}   Let $f: A_\Gamma \hookrightarrow A_{\Gamma'}$ be an injective homomorphism between two large-type Artin groups. 
    Let $\gamma, \gamma'$ be two induced cycles of $\Gamma$ that share exactly one edge $e$. Suppose that there exist  elements $g_\gamma, g_{\gamma'}\in A_{\Gamma'}$, embeddings of labelled graphs $\varphi: \gamma \hookrightarrow \Gamma'$ and $\varphi': \gamma' \hookrightarrow \Gamma'$, and integers $\varepsilon, \varepsilon' = \pm 1$ such that $f(v) = g_\gamma\varphi(v)^\varepsilon g_\gamma^{-1}$ for all standard generators $v \in V(\gamma)$ and $f(v) = g_{\gamma'}\varphi'(v)^{\varepsilon'} g_{\gamma'}^{-1}$ for all standard generators $v \in V(\gamma')$. Let $a, b$ be the vertices of the edge~$\varphi(e) \subset \Gamma'$.

    Then there exists an integer $k \in \mathbb{Z}$ such that $g_{\gamma'} = g_{\gamma} \Delta_{ab}^k$.
\end{lemma}

\begin{proof}
By hypothesis, the standard tree $\Fix(g_\gamma a g_\gamma^{-1})$ intersects each of the fundamental domains $g_{\gamma} K_{\Gamma'}$ and $g_{\gamma'} K_{\Gamma'}$ along at least an edge.
By \cite[Lemma 4.3]{martin2022acylindrical}, this means there is some $k \neq 0$ such that $g_{\gamma}^{-1}g_{\gamma'}  \in \Delta_{ab}^k \langle a \rangle$.
In particular, there is some $q \in \mathbf{Z}$ such that
\[g_{\gamma}^{-1}g_{\gamma'} = \Delta_{ab}^k a^q.\]
By the same argument but for the standard tree $\Fix(g_\gamma b g_\gamma^{-1})$, there are some $k' \neq 0$ and $q' \in \mathbf{Z}$ such that
\[g_{\gamma}^{-1}g_{\gamma'}  = \Delta_{ab}^{k'} b^{q'}.\]
Note that the two above expressions are in Garside normal forms.
By unicity of this form, it must be that $\Delta_{ab}^k = \Delta_{ab}^{k'}$ and $a^q = b^{q'}$.
This forces $k = k'$ and $q = q' = 0$.
\end{proof}

The following lemma is a direct consequence of Lemma~\ref{lem:domains_sharing_two_trees}: 

\begin{lemma}\label{lem:label_separating_edge}
    Let $\Gamma_i, \Gamma_j$ be two chunks of $\Gamma$ that share a separating edge with vertices $a, b$. Then there exists an integer $k_{i, j}\in \mathbb{Z}$ such that $\alpha_f(\Gamma_j) = \alpha_f(\Gamma_i) \Delta_{\varphi_i(a)\varphi_i(b)}^{k_{i,j}}$.
\end{lemma}

\begin{lemma}\label{lem:cohopf_induction_g_i}
     For every $i\geq 0$, we have  $\alpha_f(\Gamma_i) \in f(A_{\Gamma_0 \cup \ldots \cup \Gamma_{i-1}})$ and $\varphi_i(V(\Gamma_i)) \subseteq f(A_{\Gamma_0 \cup \ldots \cup \Gamma_i})$.
\end{lemma}

\begin{proof}
    We prove the result by induction on $i\geq 0$.  For $i=0$, there is nothing to do since $g_0=1$ by convention. For $i\geq 1$, note that since $\Gamma_i$ shares a single edge $a,b$ with $\Gamma_{j}$ for some $j<i$, we have 
   \[\alpha_f(\Gamma_i) = \alpha_f(\Gamma_j) \Delta_{\varphi_j(a)\varphi_j(b)}^{k_{j, i}}\]
   by Lemma~\ref{lem:label_separating_edge}. 
   By the induction hypothesis, we have that $\alpha_f(\Gamma_j)$ and $\varphi_j(a), \varphi_j(b)$ are in $f(A_{\Gamma_0 \cup \ldots \cup \Gamma_j})$, hence $\alpha_f(\Gamma_i) \in f(A_{\Gamma_0 \cup \ldots \cup \Gamma_j}) \subseteq f(A_{\Gamma_0 \cup \ldots \cup \Gamma_{i-1}})$. Moreover, we have 
   \[\varphi_i(v)^{\pm 1} = \alpha_f(\Gamma_i)^{-1}f(v)\alpha_f(\Gamma_i)\]
   so it follows that $\varphi_i(v) \in f(A_{\Gamma_0 \cup \ldots \cup \Gamma_i})$.
\end{proof}

\begin{lemma}
    For every $i\geq 0$, the image of the restriction $f: A_{\Gamma_0 \cup \ldots \cup \Gamma_i} \rightarrow A_\Gamma$ contains at least $|V(\Gamma_0 \cup \ldots \cup \Gamma_i)|$ standard generators of $A_\Gamma$.
\end{lemma}

\begin{proof}
    Note that for every $i$, the graph embeddings $\varphi_i: \Gamma_i \hookrightarrow \Gamma$ can be amalgamated into a graph map $\Phi_i: \Gamma_0 \cup \ldots \cup \Gamma_i \to \Gamma$.
    By Lemma~\ref{lem:cohopf_induction_g_i}, it is enough to prove that each $\Phi_i$ is injective.
    Again, we prove this by induction on $i\geq 0$. For $i=0$, this follows from Proposition~\ref{prop:inj_hom_general_case}.
    For $i\geq 1$, let $v$ be a standard generator in $\Gamma_i \backslash (\Gamma_0 \cup \ldots \cup \Gamma_{i-1})$. We have $f(v) = \alpha_f(\Gamma_i)\varphi_i(v)^{\pm 1}\alpha_f(\Gamma_i)^{-1}$ by construction. Suppose by contradiction that there exists $j<i$ such that $\varphi_i(v) = \varphi_j(w)$ for some $w \in V(\Gamma_j)$. Then, by Lemma~\ref{lem:cohopf_induction_g_i}, we have that both $\alpha_f(\Gamma_i)$ and $\varphi_j(w)$ are in $f(A_{\Gamma_0 \cup \ldots \cup \Gamma_{i-1}})$, hence so is $f(v)$. But this contradicts the injectivity of $f$, since $v\notin A_{\Gamma_0 \cup \ldots \cup \Gamma_{i-1}}$ by van der Lek \cite{van1983homotopy}. 
\end{proof}

It now follows from the previous lemma that the image of $f$ contains at least $|V(\Gamma)|$ standard generators, hence it contains a generating set of $A_\Gamma$. Thus, $f$ is surjective, and it follows that $f$ is an isomorphism, which concludes the proof of Proposition~\ref{prop:cohopf_hard_direction}.

\section{Generators of the automorphism group}\label{sec:Aut}

The goal of this final section is to describe a generating set of the automorphism group of XXXL-type Artin groups (and more generally XL-type Artin groups that satisfy the Cycle of Standard Trees Property), when the presentation graph does not have cut-vertices. 

\begin{definition}[Edge-twist \cite{brady2002rigidity}]
    Let $\Gamma$ be a simplicial graph, let $e$ be a separating edge of $\Gamma$ with vertices $a, b\in V(\Gamma)$, and let $\Gamma_1, \Gamma_2$ induced subgraphs of $\Gamma$ such that $\Gamma_1 \cup \Gamma_2 = \Gamma$ and $\Gamma_1 \cap \Gamma_2 = e$. (Note that the graphs $\Gamma_1, \Gamma_2$ may not be unique.)

    Let $\Delta_{ab}$ be the Garside element of $A_{ab}$. Conjugation by $\Delta_{ab}$ induces a permutation of $\{a, b\}$ (the identity if $m_{ab}$ is even, a transposition  otherwise). We construct a new graph $\Gamma'$ from the disjoint union $\Gamma_1 \sqcup \Gamma_2$, as follows: Let us denote by $e_1=\{a_1, b_1\}$ and $e_2=\{a_2, b_2\}$ the edges of the disjoint copies $\Gamma_1$ and $\Gamma_2$ corresponding to $e$. We identify the edges $e_1, e_2$ by identifying $a_1$ with $(a^{\Delta_{ab}})_2$, and $b_1$ with $(b^{\Delta_{ab}})_2$. The graph $\Gamma'$ is said to be obtained from $\Gamma$ by performing an \textbf{edge-twist} (on the subgraph $\Gamma_2$).
    
    We define the following isomorphism $\tau_{e, \Gamma_2}: A_{\Gamma}\rightarrow A_{\Gamma'}$, also called an \textbf{edge-twist} for simplicity, as follows: We first define it on each standard generators of the disjoint copies of $\Gamma_1, \Gamma_2$, and verify that it is compatible with the edge identification described above.
    For $v\in V(\Gamma_1)$, we set $\tau_{e, \Gamma_2}(v) := v$. For $v\in V(\Gamma_2)$, we set $\tau_{e, \Gamma_2}(v) = v^{\Delta_{ab}}$. 
\end{definition}

\begin{remark}
    These edge-twists are part of a larger family of isomorphisms between Artin groups described in \cite{brady2002rigidity}. However, for large-type Artin groups, all such isomorphisms are of the form described above.
\end{remark}

\begin{definition}
    For a presentation graph $\Gamma$, we denote by $\mathcal{T}$ the set whose elements are the presentation graphs with vertex set $V(\Gamma)$ obtained from $\Gamma$ by a sequence of edge-twists. Such a graph will be referred to as a \textbf{twist} of $\Gamma$.

    For every $\Gamma' \in \mathcal{T}$, we fix an isomorphism $\varphi_{\Gamma'}: A_{\Gamma'} \rightarrow A_\Gamma$, with the convention that~$\varphi_\Gamma = Id$. 
    For $\Gamma', \Gamma'' \in \mathcal{T}$ and an isomorphism $\varphi: A_{\Gamma'} \rightarrow A_{\Gamma''}$, we denote by 
    \[\overline \varphi := \varphi_{\Gamma''} \circ \varphi \circ \varphi_{\Gamma'}^{-1}: A_\Gamma \rightarrow A_\Gamma\]
    the associated automorphism of $A_\Gamma$.
\end{definition}

\begin{remark}
Note that the set $\mathcal{T}$ contains simplicial graphs with vertex set $V(\Gamma)$ and with the same set of labels as $\Gamma$. In particular, $\mathcal{T}$ is always finite. For instance, if $\Gamma$ is the graph from Figure \ref{FigureTriforce}, then $\mathcal{T}$ contains exactly $2^3 = 8$ elements, with the three degrees of freedom coming from the three separating edges.
\end{remark}

\begin{definition}
    Let $\Gamma', \Gamma''$ be two labelled graphs, and let $\psi: \Gamma' \rightarrow \Gamma''$ be an isomorphism of labelled simplicial graphs. Then there is an associated isomorphism of Artin groups, that we will still denote by $\psi: A_{\Gamma'} \rightarrow A_{\Gamma''}$ for simplicity, that sends each generator $v\in V(\Gamma')$ to the corresponding standard generator $\psi(v) \in V(\Gamma'')$. Such an isomorphism of Artin groups is called a \textbf{graph isomorphism}.  
\end{definition}

The goal of this section is to prove the following: 

\begin{theorem}\label{thm:Aut_generating_set}
 Let $A_\Gamma$ be an XL-type Artin group that satisfies the Cycle of Standard Trees Property. Suppose that $\Gamma$ is connected, is not an edge, and does not have a cut-vertex. Let $S_{Aut_\Gamma}$ be the finite family of automorphisms of $A_\Gamma$ consisting of: 
    \begin{itemize}
        \item the conjugations by a standard generator of $A_\Gamma$,
        \item the graph automorphisms of $A_\Gamma$, 
        \item the global inversion of $A_\Gamma$,
        \item all the $\overline \tau$ for $\tau: A_{\Gamma'} \rightarrow A_{\Gamma''}$ an edge-twist, with $\Gamma', \Gamma'' \in \mathcal{T}$. 
        \item all the $\overline \psi$ for $\psi: A_{\Gamma'} \rightarrow A_{\Gamma''}$ a graph isomorphism, with $\Gamma', \Gamma'' \in \mathcal{T}$.
    \end{itemize}
    Then $\Aut(A_\Gamma)$ is generated by $S_{Aut_\Gamma}$. In particular, $\Aut(A_\Gamma)$ is finitely generated.
\end{theorem}

\begin{remark}
    The generating set $S_{Aut_\Gamma}$ depends on the choices of isomorphisms $\varphi_{\Gamma'}$. A slightly more natural result is to show that the isomorphism groupoid of $A_\Gamma$, whose objects are the graphs $\Gamma' \in \mathcal{T}$ and whose morphisms are all the isomorphisms $A_{\Gamma'} \rightarrow A_{\Gamma''}$ for $\Gamma', \Gamma'' \in \mathcal{T}$, is generated by the finite set $S_{Iso_{\mathcal T}}$: 
    \begin{itemize}
        \item for each $\Gamma' \in \mathcal T$, the conjugations by a standard generator of $A_{\Gamma'}$,
        \item for each $\Gamma' \in \mathcal T$, the graph automorphisms of $A_{\Gamma'}$, 
        \item for each $\Gamma' \in \mathcal T$, the global inversion of $A_{\Gamma'}$,
        \item all the twists $\tau: A_{\Gamma'} \rightarrow A_{\Gamma''}$, with $\Gamma', \Gamma'' \in \mathcal{T}$. 
        \item all the graph isomorphisms $\psi: A_{\Gamma'} \rightarrow A_{\Gamma''}$, with $\Gamma', \Gamma'' \in \mathcal{T}$.
    \end{itemize}
    We leave it to the reader to check that the proofs below carry over in a straightforward way to yield this more general result. Such a result was obtained in the large-type triangle-free case by Crisp \cite{crisp2005automorphisms}.
\end{remark}

\begin{lemma}\label{lem:change_labels_chunks}
     Let $f: A_\Gamma \rightarrow A_{\Gamma'}$ be an isomorphism with $\Gamma' \in \mathcal{T}$. Let $e$ be a separating edge of $\Gamma$, and let $\Gamma_0, \ldots, \Gamma_k$ denote the family of chunks of $\Gamma$ containing $e$.
     Let us assume that $\alpha_f(\Gamma_0)=1$, where the coefficients $\alpha_f(\Gamma_i)$ were defined in Definition \ref{def:labelling_decomposition_tree}. 
     Then there exists a sequence $s_1, \ldots, s_n$ of edge-twists in $S_{Iso_{\mathcal T}}$ and a graph $\Gamma''$ in $\mathcal T$ (obtained from $\Gamma'$ by a sequence of edge-twists) such that 
    the composition 
    \[f' := s_n \circ \cdots \circ s_1 \circ f: A_\Gamma \rightarrow A_{\Gamma''}\]
    is well-defined, and this isomorphism satisfies 
     \[\alpha_{f'}(\Gamma_i) = 1 ~~~\mbox{ for all } ~~~ 0\leq i \leq k\]
\end{lemma}

\begin{proof}
 By applying Lemma~\ref{lem:induced_isom_chunks} to each restriction $f: A_{\Gamma_i}\rightarrow A_{\Gamma'}$, we get an embedding of labelled graphs $\varphi_i: \Gamma_i \hookrightarrow \Gamma'$. Note that since each $\Gamma_i$ is a chunk, the image $\varphi_i(\Gamma_i)$ is contained in a chunk $\Gamma_i'$ of $\Gamma'$. Moreover, by applying the same reasoning to the inverse isomorphism $f^{-1}$, we have that $\Gamma_i' \neq \Gamma_j'$ for each $i\neq j$. 

Denote by $e' := \varphi_0(e) = \cdots = \varphi_k(e)$, and let us denote by $a, b \in V(\Gamma)$ the vertices of $e'$. 
    Since $\alpha_f(\Gamma_0)=1$, for each $1\leq i \leq k$, Lemma~\ref{lem:domains_sharing_two_trees} yields an integer $m_i\in \mathbb{Z}$ such that $\alpha_f(\Gamma_i) = \Delta_{ab}^{m_i}$.
    The composition 
   \[f' := \tau_{e', \Gamma_1'}^{-m_1}\circ \cdots \tau_{e', \Gamma_k'}^{-m_k}\circ f\] thus satisfies 
   $\alpha_{f'}(\Gamma_i) = 1  = \alpha_f(\Gamma_0)$ for all $ 0\leq i \leq k$.
\end{proof}

\begin{corollary}\label{cor:twists_make_trivial_labels}
    Let $f: A_{\Gamma} \rightarrow A_\Gamma$ be an automorphism. There exists a sequence of conjugations $i_1, \ldots, i_k$ in $S_{Aut_{\mathcal{T}}}$, a sequence $s_1, \ldots, s_n$ of edge-twists in $S_{Iso_{\mathcal T}}$ and a graph $\Gamma'$ in $\mathcal T$ such that 
    the composition 
    \[f' := s_n \circ \cdots \circ s_1 \circ i_k \circ \cdots \circ i_1 \circ f: A_\Gamma \rightarrow A_{\Gamma'}\]
    is well-defined, and this isomorphism is such that $\alpha_{f'}(\Gamma_i) =1$ for every chunk $\Gamma_i$ of $\Gamma$. 
\end{corollary}

\begin{proof}
    Up to post-conjugation by $\alpha_{f}(\Gamma_0)^{-1}$ (which amounts to post-composing $f$ by a finite number of conjugations  in $S_{Aut_\Gamma}$), we can assume that $\alpha_f(\Gamma_0)=1$. One now proves the result by induction on the number of chunks of $\Gamma$:  The case of $\Gamma$ having one single chunk is a direct consequence of Lemma~\ref{lem:induced_isom_chunks}, and since the chunk tree is connected, the induction step is obtained by applying Lemma~\ref{lem:change_labels_chunks}. 
\end{proof}

\begin{lemma}\label{lem:iso_trivial_labels}
    Let $f: A_{\Gamma} \rightarrow A_{\Gamma'}$ be an isomorphism such that $\alpha_f(\Gamma_i)=1$ for every chunk $\Gamma_i$ of $\Gamma$. Then $\Gamma'$ is isomorphic to $\Gamma$ as a labelled graph.

    Moreover, let $f: A_{\Gamma} \rightarrow A_{\Gamma}$ be an automorphism such that $\alpha_f(\Gamma_i)=1$ for every chunk $\Gamma_i$ of $\Gamma$. Then $f$ is in the subgroup of $\Aut(A_\Gamma)$ generated by the graph automorphisms and the global inversion of $A_\Gamma$.
\end{lemma}

\begin{proof}
    By Lemma~\ref{lem:induced_isom_chunks} applied to each chunk $\Gamma_i$, and using the fact that $\alpha_f(\Gamma_i)=1$, we have that for each $i$, there exists an embedding $\varphi_i: \Gamma_i \hookrightarrow \Gamma'$ of labelled graphs, and an integer $\varepsilon_i = \pm 1$ such that for every $v\in V(\Gamma_i)$, we have 
    \[f(v) = \varphi_i(v)^{\varepsilon_i}\]
    Since the graph $T$ is connected, it follows that the $\varepsilon_i$ are all equal. Up to possibly post-composing $f$ with the global inversion of $A_{\Gamma'}$ (which is in $S_{Aut}$), we can assume that $\varepsilon_i =1$ for all $i$. By injectivity of $f$, we get that the $f(v)$ are all distinct, hence the embeddings $\varphi_i$ can be amalgamated to an embedding of labelled graphs $\varphi: \Gamma \hookrightarrow \Gamma'$. Since $f$ is an isomorphism by assumption, it follows that $\varphi$ is surjective, hence an isomorphism. 
    
    If $f: A_\Gamma \rightarrow A_\Gamma$ is an automorphism, the same reasoning as above for $\Gamma' = \Gamma$ implies that there exists an integer $\varepsilon = \pm 1$ and an isomorphism of labelled graphs $\varphi: \Gamma \rightarrow \Gamma$ such that for every $v\in V(\Gamma)$, we have 
    \[f(v) = \varphi(v)^{\varepsilon}\]
    and the result follows.  
\end{proof}

\begin{corollary}\label{cor:auto_simplification}
    Let $A_\Gamma$ be an XL-type Artin group that satisfies the Cycle of Standard Trees Property. Further assume that $\Gamma$ is connected, with $|V(\Gamma)|\geq 3$, and without cut vertex. Let $f: A_{\Gamma} \rightarrow A_\Gamma$ be an automorphism. Then there exists a sequence of conjugations $i_1, \ldots, i_k$ in $S_{Aut_{\mathcal{T}}}$, a sequence $s_1, \ldots, s_n$ of edge-twists in $S_{Iso_{\mathcal T}}$ and a graph isomorphism $\psi$ in $S_{Iso_{\mathcal T}}$ such that 
    the composition 
    \[f' := \psi \circ s_n \circ \cdots \circ s_1 \circ i_k \circ \cdots \circ i_1 \circ f: A_\Gamma \rightarrow A_{\Gamma}\]
    is a well-defined automorphism of $A_\Gamma$, and is in the subgroup of $\Aut(A_\Gamma)$ generated by the conjugations, graph automorphisms, and the global inversion of $A_\Gamma$.
\end{corollary}

\begin{proof}
    This is a direct consequence of Corollary~\ref{cor:twists_make_trivial_labels} and Lemma~\ref{lem:iso_trivial_labels}.
\end{proof}

\begin{proof}[Proof of Theorem~\ref{thm:Aut_generating_set}]
Let $f: A_\Gamma \rightarrow A_\Gamma$ be an automorphism. By Corollary~\ref{cor:auto_simplification}, we can pick edge-twists $s_1, \ldots, s_n \in S_{Iso_{\mathcal T}}$, and a graph isomorphism $\psi\in S_{Iso_{\mathcal T}}$ such that the composition
\[\psi \circ s_n \circ \cdots \circ s_1 \circ i_k \circ \cdots \circ i_1 \circ f: A_\Gamma \rightarrow A_{\Gamma}\]
is well-defined, and is in the subgroup of $\Aut(A_\Gamma)$ generated by the conjugations, graph automorphisms, and the global inversion of $A_\Gamma$.

By construction of the isomorphisms $\varphi_{\Gamma'}$ for $\Gamma' \in \mathcal{T}$, we can now rewrite 
\[\psi \circ s_n \circ \cdots \circ s_1 \circ i_k \circ \cdots \circ i_1 \circ f = \overline{\psi} \circ \overline{s_n} \circ \cdots \circ \overline{s_1} \circ i_k \circ \cdots \circ i_1 \circ f\]
with $\overline{\psi}, \overline{s_1}, \ldots, \overline{s_n} \in S_{Aut_\Gamma}$ by construction. Thus, $f$ can be written as a product of elements of $S_{Aut_\Gamma}$, and we are done.
\end{proof}

\bibliographystyle{alpha}
\bibliography{mybib}

\vspace{0.5cm}

\bigskip\noindent
\textbf{Mart\'in Blufstein}, 

\noindent Address: Department of Mathematical Sciences, University of Copenhagen, 2100 Copenhagen, Denmark.

\noindent Email: \texttt{mblufstein@dm.uba.ar}

\smallskip

\bigskip\noindent
\textbf{Alexandre Martin}, 

\noindent Address: Department of Mathematics and the Maxwell Institute for the Mathematical Sciences, Heriot-Watt University, Edinburgh EH14 4AS, UK.

\noindent Email: \texttt{alexandre.martin@hw.ac.uk}

\smallskip

\bigskip\noindent
\textbf{Nicolas Vaskou}, 

\noindent Address: School of Mathematics, University of Bristol, Bristol BS8 1UG, UK.

\noindent Email: \texttt{nicolas.vaskou@gmail.com}

\end{document}